\documentclass[11pt,a4paper]{article}

\usepackage{amsmath}
\usepackage{amsthm}
\usepackage{amssymb}
\usepackage{amscd}
\usepackage[all]{xy}

\title{Decomposition of De Rham Complexes with Smooth Horizontal
Coefficients for Semistable Reductions \footnote{This paper was
partially supported by the National Natural Science Foundation
of China (Grant No.\ 10901037), the Ph.D.\ Programs Foundation
of Ministry of Education of China (Grant No.\ 20090071120004)
and the Scientific Research Foundation for the Returned
Overseas Chinese Scholars, State Education Ministry.}}
\author{Qihong Xie}
\date{}
\theoremstyle{plain}
\newtheorem{prop}{Proposition}[section]
\newtheorem{lem}[prop]{Lemma}
\newtheorem{thm}[prop]{Theorem}
\newtheorem{cor}[prop]{Corollary}
\newtheorem{conj}[prop]{Conjecture}

\theoremstyle{definition}
\newtheorem{defn}[prop]{Definition}
\newtheorem{asp}[prop]{Assumption}
\newtheorem*{nota}{Notation}
\newtheorem*{ack}{Acknowledgments}

\theoremstyle{remark}
\newtheorem{rem}[prop]{Remark}

\newcommand{\Q}{\mathbb Q}

\newcommand{\Z}{\mathbb Z}

\newcommand{\F}{\mathbb F}
\newcommand{\A}{\mathbb A}
\newcommand{\PP}{\mathbb P}
\newcommand{\OO}{\mathcal O}

\newcommand{\MM}{\mathcal M}
\newcommand{\HH}{\mathcal H}
\newcommand{\LL}{\mathcal L}

\newcommand{\ZZ}{\mathcal Z}
\newcommand{\UU}{\mathcal U}
\newcommand{\CC}{\mathcal C}

\newcommand{\Supp}{\mathop{\rm Supp}\nolimits}
\newcommand{\Ext}{\mathop{\rm Ext}\nolimits}
\newcommand{\ch}{\mathop{\rm char}\nolimits}

\newcommand{\Spec}{\mathop{\rm Spec}\nolimits}

\newcommand{\im}{\mathop{\rm Im}\nolimits}
\newcommand{\divisor}{\mathop{\rm div}\nolimits}
\newcommand{\gr}{\mathop{\rm gr}\nolimits}
\newcommand{\Fil}{\mathop{\rm Fil}\nolimits}
\newcommand{\sgn}{\mathop{\rm sgn}\nolimits}
\newcommand{\ra}{\rightarrow}

\newcommand{\wt}{\widetilde}
\newcommand{\bR}{{\bf R}}
\newcommand{\cH}{{\bf H}}
\newcommand{\bH}{\mathbb H}

\numberwithin{equation}{section}

\begin{document}
\maketitle

\begin{abstract}
We generalize Illusie's result to prove the decomposition of the de
Rham complex with smooth horizontal coefficients for a semistable
$S$-morphism $f:X\ra Y$ which is liftable over $\Z/p^2\Z$.
As an application, we prove the Koll\'ar vanishing theorem in positive
characteristic for a semistable $S$-morphism $f:X\ra Y$ which is liftable
over $\Z/p^2\Z$, where all concerned horizontal divisors are smooth over $Y$.
\end{abstract}

\setcounter{section}{0}
\section{Introduction}\label{S0}

The decomposition of de Rham complexes is one of the most important
results in algebraic geometry of positive characteristic, which has
been discovered by Deligne and Illusie \cite{di} and successfully
used to give a purely algebraic proof of the Kodaira vanishing
theorem. More precisely, let $k$ be a perfect field of characteristic
$p>0$, and $W_2(k)$ the ring of Witt vectors of length two of $k$.
Let $S$ be a $k$-scheme, $\wt{S}$ a lifting of $S$ over $W_2(k)$,
$X$ a smooth $S$-scheme, and $F:X\ra X_1$ the relative Frobenius
morphism of $X$ over $S$. If $X$ has a lifting $\wt{X}$ over $\wt{S}$
and $\dim(X/S)<p$, then we have a decomposition in $D(X_1)$:
\begin{eqnarray*}
\bigoplus_i\Omega^i_{X_1/S}[-i]\stackrel{\sim}{\ra} F_*\Omega^\bullet
_{X/S}.
\end{eqnarray*}

Later, Illusie \cite{il90} has generalized the above result to the
relative case for a semistable $S$-morphism $f:X\ra Y$ to obtain
the decomposition of de Rham complexes with coefficients in the
Gauss-Manin systems. Roughly speaking, let $E$ be the branch divisor
of $f$, $D=X\times_Y E$, and $\bH=\oplus_i\bR^if_*\Omega^\bullet
_{X/Y}(\log D/E)$ the Gauss-Manin system. If $f:X\ra Y$ has a lifting
$\wt{f}:\wt{X}\ra \wt{Y}$ over $\wt{S}$ and $\dim(X/S)<p$, then
we have a decomposition in $D(Y_1)$:
\begin{eqnarray*}
\bigoplus_i\gr^i\Omega^\bullet_{Y_1/S}(\log E_1)(\bH_1)\stackrel{\sim}
{\ra} F_*\Omega^\bullet_{Y/S}(\log E)(\bH).
\end{eqnarray*}

In this paper, we generalize Illusie's result to the case where
smooth horizontal coefficients are taken into account. Roughly speaking,
let $D$ be an adapted divisor on $X$, i.e.\ $D$ consists of three parts:
all singular fibers of $f$, some smooth fibers of $f$ and some smooth
horizontal divisors with respect to $f$ (see Definition \ref{1.2} for
more details). Let $\bH=\oplus_i\bR^if_*\Omega^\bullet_{X/Y}(\log D/
E_a)$. Then we prove that if $f:(X,D)\ra (Y,E_a)$ has a lifting
$\wt{f}:(\wt{X},\wt{D})\ra (\wt{Y},\wt{E}_a)$ and $\dim(X/S)<p$, then
there is a decomposition in $D(Y_1)$ (see Theorem \ref{5.9} for more
details):
\begin{eqnarray*}
\bigoplus_i\gr^i\Omega^\bullet_{Y_1/S}(\log E_{a1})(\bH_1)\stackrel
{\sim}{\ra} F_*\Omega^\bullet_{Y/S}(\log E_a)(\bH),
\end{eqnarray*}
from which follows the Koll\'ar vanishing theorem in positive
characteristic, i.e.\
\begin{eqnarray*}
H^i(Y,\LL\otimes R^jf_*\omega_{X/S}(D))=0
\end{eqnarray*}
holds for any $i>0,j\geq 0$ and any ample invertible sheaf $\LL$
on $Y$ (see Theorem \ref{6.3} for more details). It should be mentioned
that the proofs of all of the results in this paper follow Illusie's
arguments very closely.

In general, we may put forward the following conjecture, called the
logarithmic Koll\'ar vanishing for semistable reductions in positive
characteristic (see \cite[Theorem 10.19]{ko95} for the logarithmic
Koll\'ar vanishing theorem in characteristic zero):

\begin{conj}\label{0.1}
Let $X$ and $Y$ be proper and smooth $S$-schemes, $f:X\ra Y$ an
$E$-semistable $S$-morphism, and $D$ a simple normal crossing divisor
on $X$ containing the divisor $X\times_Y E$. Let $H$ be a $\Q$-divisor
on $X$ such that the support of the fractional part of $H$ is contained
in $D$ and $H\sim_\Q f^*L$, where $L$ is an ample $\Q$-divisor on $Y$.
Assume that $f:(X,D)\ra(Y,E)$ has a lifting $\wt{f}:(\wt{X},\wt{D})
\ra(\wt{Y},\wt{E})$ over $\wt{S}$ and $\dim(X/S)<p$.
Then $H^i(Y,R^jf_*\OO_X(K_{X/S}+\ulcorner H\urcorner))=0$ holds for
any $i>0$ and $j\geq 0$.
\end{conj}

There are several difficulties in dealing with Conjecture \ref{0.1}.
First, we need some technique to change the $\Q$-divisor argument
into the integral divisor argument. Second, the situation of the
horizontal divisors contained in $D$ or $H$ is complicated. Third,
the decomposition of de Rham complexes with horizontal coefficients is
completely unknown. In this sense, all of results obtained in this paper
may be regarded as the first step to resolving Conjecture \ref{0.1}.

\begin{nota}
We denote the support of a divisor $D$ by $\Supp(D)$, the relative
dualizing sheaf of $f:X\ra Y$ by $\omega_{X/Y}$, and the divisor
defined by $x=0$ by $\divisor_0(x)$.
\end{nota}

\begin{ack}
I would like to express my gratitude to Professor Luc Illusie for
invaluable advices and warm encouragement. I am grateful to the referee
for pointing out some inaccurate places and giving many useful suggestions,
which make this paper more readable.
\end{ack}

\tableofcontents

\section{Definitions and preliminaries}\label{S1}

This section is parallel to \cite[\S 1]{il90}, and all proofs follow
Illusie's proofs very closely.

\begin{defn}\label{1.1}
Let $S$ be a scheme, $X$ and $Y$ smooth $S$-schemes, and $f:X\ra Y$
an $S$-morphism. Let $E\subset Y$ be a divisor relatively simple normal
crossing over $S$ (RSNC for short), and $E_X=X\times_Y E$.
We say that $f:X\ra Y$ is $E$-semistable, or that $f$ has a semistable
reduction along $E$, if locally for the \'etale topology, $f$ is the
product of $S$-morphisms of the following type (i) or (ii):
\begin{itemize}
\item[(i)] $pr_1:\A^n_S\ra\A^1_S$, $E=\emptyset$;
\item[(ii)] $h:\A^n_S\ra\A^1_S$, $h^*y=x_1\cdots x_n$, where
$\A^n_S=\Spec\OO_S[x_1,\cdots,x_n]$, $\A^1_S=\Spec\OO_S[y]$,
and $E=\divisor_0(y)$.
\end{itemize}
\end{defn}

\begin{defn}\label{1.2}
Let $f:X\ra Y$ be an $E$-semistable morphism as in Definition
\ref{1.1}. A divisor $D\subset X$ is said to be adapted to $f$,
if the following conditions hold:
\begin{itemize}
\item[(i)] $D$ admits a decomposition of irreducible components:
$D=E_X+D_a+D_h$, where $D_a$ is the sum of the irreducible components of
$D$ whose images under $f$ are divisors not contained in $E$, and $D_h$ is
the sum of the irreducible components of $D$ whose images under $f$ are the
whole $Y$; and
\item[(ii)] $D$ is RSNC over $S$, $D_h$ is RSNC over $Y$, and the union of
the divisor $A:=f(D_a)$ and $E$ is RSNC over $S$.
\end{itemize}
\end{defn}

\begin{rem}\label{1.3}
(1) The divisor $E_X$ is adapted to $f$.

(2) In Definition \ref{1.2}, for any irreducible component $D_{h1}$ of
$D_h$, the restriction morphism $f|_{D_{h1}}:D_{h1}\ra Y$ is smooth.
\end{rem}

\begin{defn}\label{1.4}
Let $f:X\ra Y$ be an $E$-semistable $S$-morphism with an adapted divisor
$D$ as in Definition {1.2}. For simplicity, denote $E+A$ by $E_a$.
Let $\Omega_{X/S}^\bullet(\log D)$ (resp.\ $\Omega_{Y/S}^\bullet(\log E_a)$)
be the de Rham complex of $X$ (resp.\ $Y$) with logarithmic poles along $D$
(resp.\ $E_a$). We define $\Omega_{X/Y}^\bullet(\log D/E_a)$, the de Rham
complex of $X$ over $Y$ with relative logarithmic poles along $D$ over $E_a$,
in the following way:

Let $\Omega_{X/Y}^1(\log D/E_a)$ be the quotient of $\Omega_{X/S}^1(\log D)$
by the image of $f^*\Omega_{Y/S}^1(\log E_a)$. Then by Lemma \ref{1.5},
$\Omega_{X/Y}^1(\log D/E_a)$ is a locally free sheaf on $X$ of rank $d=n-e$,
where $n=\dim(X/S)$ and $e=\dim(Y/S)$. Let $\Omega_{X/Y}^i(\log D/E_a)=
\wedge^i\Omega_{X/Y}^1(\log D/E_a)$ and define the differential $d$ by
passage to the quotient of that of the complex $\Omega_{X/S}^\bullet(\log D)$.
\end{defn}

It is easy to see that if $f$ is smooth, then $\Omega_{X/Y}^\bullet(\log D/E_a)
=\Omega_{X/Y}^\bullet(\log D_h)$ in the usual sense.

\begin{lem}\label{1.5}
With notation as in Definition \ref{1.4},
there is an exact sequence of locally free sheaves of finite type:
\begin{equation}
0\ra f^*\Omega_{Y/S}^1(\log E_a)\ra \Omega_{X/S}^1(\log D)\ra
\Omega_{X/Y}^1(\log D/E_a)\ra 0. \label{f1.1}
\end{equation}
\end{lem}

\begin{proof}
We have only to prove the statement locally for the \'etale topology,
so it suffices to check for the following three types, where $\A^n_S=
\Spec\OO_S[x_1,\cdots,x_n]$ and $\A^1_S=\Spec\OO_S[y]$:
\begin{itemize}
\item[(i)] $pr_1:\A^n_S\ra\A^1_S$, where $E=\emptyset$, $A=\emptyset$,
$pr^*_1(y)=x_1$, and $D_h=\divisor_0(x_2\cdots x_r)$ for $1\leq r\leq n$;
\item[(ii)] $pr_1:\A^n_S\ra\A^1_S$, where $E=\emptyset$, $A=\divisor_0(y)$,
$pr^*_1(y)=x_1$, $D_a=\divisor_0(x_1)$, and $D_h=\divisor_0(x_2\cdots x_r)$
for $1\leq r\leq n$;
\item[(iii)] $h:\A^n_S\ra\A^1_S$, where $E=\divisor_0(y)$, $h^*y=x_1
\cdots x_s$, $A=\emptyset$, and $D_h=\divisor_0(x_{s+1}\cdots x_{n})$
for $2\leq s\leq n$.
\end{itemize}

(i) $f^*\Omega_{Y/S}^1(\log E_a)$ is generated by $f^*(dy)=dx_1$,
$\Omega_{X/S}^1(\log D)$ is generated by $dx_1,dx_2/x_2,\cdots,dx_r/x_r,
dx_{r+1},\cdots,dx_n$, hence $\Omega_{X/Y}^1(\log D/E_a)$ is generated by \\
$dx_2/x_2,\cdots,dx_r/x_r,dx_{r+1},\cdots,dx_n$, so the conclusion is clear.

(ii) $f^*\Omega_{Y/S}^1(\log E_a)$ is generated by $f^*(dy/y)=dx_1/x_1$,
$\Omega_{X/S}^1(\log D)$ is generated by $dx_1/x_1,\cdots,dx_r/x_r,
dx_{r+1},\cdots,dx_n$, hence $\Omega_{X/Y}^1(\log D/E_a)$ is generated by
$dx_2/x_2,\cdots,dx_r/x_r,dx_{r+1},\cdots,dx_n$, so the conclusion is clear.

(iii) $f^*\Omega_{Y/S}^1(\log E_a)$ is generated by $f^*(dy/y)=
\sum_{i=1}^s dx_i/x_i$, $\Omega_{X/S}^1(\log D)$ is generated by
$dx_1/x_1,\cdots,dx_n/x_n$, hence
\[ \Omega_{X/Y}^1(\log D/E_a)=\OO_X\langle\frac{dx_1}{x_1},\cdots,
\frac{dx_s}{x_s}\rangle/(\sum_{i=1}^s\frac{dx_i}{x_i})\bigoplus
\OO_X\langle\frac{dx_{s+1}}{x_{s+1}},\cdots,\frac{dx_n}{x_n}\rangle, \]
so the conclusion is clear.
\end{proof}

\begin{rem}\label{1.6}
(1) Let $f_i:X_i\ra Y_i$ be $E_i$-semistable $S$-morphisms with adapted
divisors $D_i$ as in Definition \ref{1.2}. Let $X$, $Y$ and $f:X\ra Y$
be the external products over $S$ of $X_i$, $Y_i$ and $f_i$ respectively.
Then $\Omega_{X/Y}^\bullet(\log D/E_a)$ is the external tensor product of
$\Omega_{X_i/Y_i}^\bullet(\log D_i/E_{ai})$ over $S$.

(2) In the exact sequence (\ref{f1.1}), taking the top exterior tensor
product gives rise to the following canonical isomorphism:
$f^*\Omega_{Y/S}^e(\log E_a)\otimes\Omega_{X/Y}^d(\log D/E_a)\stackrel
{\sim}{\ra}\Omega_{X/S}^n(\log D)$. Since $\Omega_{Y/S}^e(\log E_a)=
\omega_{Y/S}(E_a)$, $\Omega_{X/S}^n(\log D)=\omega_{X/S}(D)$, we have
$\Omega_{X/Y}^d(\log D/E_a)\cong\omega_{X/Y}(D_h)$.

(3) Let $f':X'\ra Y'$ be deduced from $f:X\ra Y$ by a base change
$Y'\ra Y$. Pose
\begin{equation}
\Omega_{X'/Y'}^\bullet(\log D'/E'_a)=
\Omega_{X/Y}^\bullet(\log D/E_a)\otimes_{\OO_X}\OO_{X'}. \label{f1.2}
\end{equation}
Note that, in general, $X'$ is no longer smooth over $S$, and that it is
no longer possible to interpret $\Omega_{X'/Y'}^\bullet(\log D'/E'_a)$
as a de Rham complex with relative logarithmic poles.

(4) Let $j:U\hookrightarrow X$ be the open subset over which $f$ is smooth.
Then we have a canonical isomorphism:
\begin{equation}
\Omega_{X/Y}^\bullet(\log D/E_a)\stackrel{\sim}{\longrightarrow}
j_*\Omega_{U/Y}^\bullet(\log D|_U/E_a). \label{f1.3}
\end{equation}
In fact, for any point $s\in S$, $X_s-U_s$ is of codimension at least 2
in $X_s$, therefore $\Omega_{X/Y}^\bullet(\log D/E_a)$ is the unique
prolongation of $\Omega_{U/Y}^\bullet(\log D|_U/E_a)$ with components
being locally free of finite type.
\end{rem}

From now on, let $S$ be a scheme of characteristic $p>0$, and $f:X\ra Y$
an $E$-semistable $S$-morphism with an adapted divisor $D$ as in Definition
\ref{1.2}. Let $F_X$ and $F_Y$ be the absolute Frobenius morphisms
of $X$ and $Y$, which fit into the following commutative diagram:
\begin{equation}
\xymatrix{
X \ar[dr]_f \ar[r]^F & X' \ar[d]^{f'} \ar[r] & X \ar[d]^{f} \\
                     & Y  \ar[r]^{F_Y}       & Y,
} \label{f1.4}
\end{equation}
where the square is cartesian and the composition of the upper
horizontal morphisms is equal to $F_X$.

The differential $d$ of the complex $F_*\Omega_{X/Y}^\bullet(\log D/E_a)$
is $\OO_{X'}$-linear, so we would like to calculate its cohomology
$\OO_{X'}$-modules by a Cartier type isomorphism. Consider the following
commutative diagram with cartesian square:
\begin{equation}
\xymatrix{
X \ar[dr] \ar[r]^{F_{X/S}} & X_1 \ar[d] \ar[r] & X \ar[d] \\
                        & S  \ar[r]^{F_S}   & S,
} \label{f1.5}
\end{equation}
where $F_{X/S}:X\ra X_1$ is the relative Frobenius morphism of $X$ over $S$.

By \cite[7.2]{katz}, we have the Cartier isomorphism
\[ C^{-1}:\Omega^1_{X_1/S}(\log D_1)\stackrel{\sim}{\ra}
\HH^1(F_{X/S*}\Omega_{X/S}^\bullet(\log D)), \]
where $D_1$ is the pullback of $D$ by $F_S$. By adjunction of
$(F_S^*,F_{S*})$ and abuse of notation, we have the homomorphism
\[ C^{-1}:\Omega^1_{X/S}(\log D)\ra\HH^1(F_{X*}\Omega_{X/S}^
\bullet(\log D)), \]
which sends $dx$ (resp.\ $dx/x$) to the cohomology class of
$x^{p-1}dx$ (resp.\ $dx/x$) in the $\OO_X$-module
$\HH^1(F_{X*}\Omega_{X/S}^\bullet(\log D))$.
The natural surjective morphism of complexes of $\OO_X$-modules
$F_{X*}\Omega_{X/S}^\bullet(\log D)\ra F_{X*}\Omega_{X/Y}^\bullet
(\log D/E_a)$ induces a natural homomorphism
\[ \pi:\HH^1(F_{X*}\Omega_{X/S}^\bullet
(\log D))\ra\HH^1(F_{X*}\Omega_{X/Y}^\bullet(\log D/E_a)), \]
which kills all cohomology classes of $y^{p-1}dy$ (resp.\ $dy/y$), where
$dy$ (resp.\ $dy/y$) are local sections of $f^*\Omega^1_{Y/S}(\log E_a)$.
The composition
\[ \pi\circ C^{-1}:
\Omega^1_{X/S}(\log D)\ra\HH^1(F_{X*}\Omega_{X/Y}^\bullet(\log D/E_a)) \]
vanishes on $f^*\Omega^1_{Y/S}(\log E_a)$, which defines the homomorphism
\[ C^{-1}:\Omega^1_{X/Y}(\log D/E_a)\ra\HH^1(F_{X*}\Omega_{X/Y}^\bullet
(\log D/E_a)). \]
By adjunction of $(F_Y^*,F_{Y*})$, we have the Cartier homomorphism
\[ C^{-1}:\Omega^1_{X'/Y}(\log D'/E_a)\ra\HH^1(F_*\Omega_{X/Y}^\bullet
(\log D/E_a)). \]
The exterior product gives rise to a homomorphism of graded
$\OO_{X'}$-algebras:
\begin{equation}
C^{-1}:\Omega^*_{X'/Y}(\log D'/E_a)\ra\HH^*(F_*\Omega_{X/Y}^\bullet
(\log D/E_a)). \label{f1.6}
\end{equation}

\begin{prop}\label{1.7}
The homomorphism (\ref{f1.6}) is an isomorphism.
\end{prop}

\begin{proof}
Since (\ref{f1.6}) is compatible with \'etale topology and external tensor
products over $S$, it suffices to prove the statement for those three types
described as in the proof of Lemma \ref{1.5}.

(i) and (ii) In these cases, $f$ is smooth, hence
$\Omega_{X/Y}^\bullet(\log D/E_a)=\Omega_{X/Y}^\bullet(\log D_h)$.
Thus (\ref{f1.6}) is just the usual Cartier isomorphism
\cite[7.2]{katz}.

(iii) In this case, we can further assume $S=\Spec\F_p$.
The diagram (\ref{f1.4}) corresponds to the following diagram of rings:
\[
\xymatrix{
B \ar[r] & B' \ar[r]^F & B \\
A \ar[u]^{f^*} \ar[r]^{F_A} & A, \ar[u]^{f^{\prime *}} \ar[ur]_{f^*}
}
\]
where $A=\F_p[y]$, $B=\F_p[x_1,\cdots,x_n]$, $f^*(y)=x_1\cdots x_s$,
$F_A(y)=y^p$ and $F(x_i)=x_i^p$. If we identify $B$ with
the $A$-algebra $\F_p[x_1,\cdots,x_n,y]/(y-x_1\cdots x_s)$, then
$B'$ can be identified with the $A$-algebra $\F_p[x_1,\cdots,x_n,y]/
(y^p-x_1\cdots x_s)$ since $y\in A$ is sent to $y^p$. Thus $B'$
can also be identified with the $A$-algebra $\F_p[x_1^p,\cdots,x_s^p,
x_{s+1},\cdots,x_n,x_1\cdots x_s]$. Define $B_1=\F_p[x_1,\cdots,x_s]$,
$B_2=\F_p[x_{s+1},\cdots,x_n]$, $B'_1=\F_p[x_1^p,\cdots,x_s^p,x_1\cdots
x_s]$, and $B'_2=\F_p[x_{s+1},\cdots,x_n]$. Then $B=B_1\otimes B_2$,
$B'=B'_1\otimes B'_2$, $F:B'\ra B$ factorizes into the external tensor
product of $F_1:B'_1\ra B_1$ defined by the inclusion and $F_2:B'_2
\ra B_2$ defined by the $p$-th power map, and $B\ra B'$ factorizes into
the external tensor product of $B_j\ra B'_j$ for $j=1,2$, where $B_1
\ra B'_1$ is defined by $x_i\mapsto x_i^p$ ($1\leq i\leq s$) and
$B_2\ra B'_2$ is defined by $x_i\mapsto x_i$ ($s+1\leq i\leq n$).
To prove that (\ref{f1.6}) is an isomorphism,
it suffices to prove that
\begin{equation}
C^{-1}:\Omega^*_{B'_j/A}(\log D'/E_a)\ra \HH^*(F_*\Omega_{B_j/A}
^\bullet(\log D/E_a)) \label{f1.7}
\end{equation}
is an isomorphism for $j=1,2$.
When $j=1$, it was proved in \cite[Proposition 1.5]{il90}.
When $j=2$, (\ref{f1.7}) is just the usual Cartier isomorphism
\cite[7.2]{katz}.
\end{proof}

\begin{rem}
Note that in case (iii), $f$ is no longer smooth, $X'$ is no longer
smooth over $S$ and $F:X\ra X'$ is no longer flat.
\end{rem}

\section{Decomposition of de Rham complex with relative logarithmic
poles}\label{S2}

This section is parallel to \cite[\S 2]{il90}, and all proofs follow
Illusie's proofs very closely.

\begin{defn}\label{2.1}
Let $S$ be a scheme of characteristic $p>0$. A lifting of $S$ over $\Z/p^2\Z$
is a scheme $\wt{S}$, defined and flat over $\Z/p^2\Z$ such that $\wt{S}\times
_{\Spec\Z/p^2\Z}\Spec\F_p=S$. A lifting of the absolute Frobenius morphism
$F_S:S\ra S$ over $\wt{S}$ is an endomorphism $F_{\wt{S}}:\wt{S}\ra\wt{S}$ of
$\wt{S}$ such that $F_{\wt{S}}|_S=F_S$. A lifting of an $E$-semistable $S$-morphism
$f:X\ra Y$ with an adapted divisor $D$ over $\wt{S}$ is an $\wt{E}$-semistable
$\wt{S}$-morphism $\wt{f}:\wt{X}\ra\wt{Y}$ with an adapted divisor $\wt{D}$
as in Definition \ref{1.2}, such that $\wt{X}\times_{\wt{S}}S=X$,
$\wt{Y}\times_{\wt{S}}S=Y$, $\wt{D}\times_{\wt{S}}S=D$,
$\wt{E}\times_{\wt{S}}S=E$ and $\wt{f}|_X=f$. We say that
$\wt{f}:(\wt{X},\wt{D})\ra (\wt{Y},\wt{E}_a)$ is a lifting of
$f:(X,D)\ra (Y,E_a)$ over $\wt{S}$, if no confusion is likely.
\end{defn}

In this section, let $S$ be a scheme of characteristic $p>0$, $\wt{S}$
a lifting of $S$ over $\Z/p^2\Z$, and $F_{\wt{S}}:\wt{S}\ra\wt{S}$ a lifting
of the absolute Frobenius morphism $F_S:S\ra S$ over $\wt{S}$.
Let $f:X\ra Y$ be an $E$-semistable $S$-morphism with an adapted divisor $D$
as in Definition \ref{1.2}, and $\wt{f}:(\wt{X},\wt{D})\ra (\wt{Y},\wt{E}_a)$
a lifting of $f:(X,D)\ra (Y,E_a)$ over $\wt{S}$ as in Definition \ref{2.1}.
Let $\wt{D}_1\subset\wt{X}_1$ be the $\wt{S}$-schemes deduced from
$\wt{D}\subset\wt{X}$ by the base change $F_{\wt{S}}$, and
$\wt{F}:\wt{X}\ra\wt{X}_1$ an $\wt{S}$-morphism lifting
the relative Frobenius morphism $F:X\ra X_1$ of $X$ over $S$.
\[
\xymatrix{
\wt{X} \ar[dr] \ar[r]^{\wt{F}} & \wt{X}_1 \ar[d] \ar[r] & \wt{X} \ar[d] \\
                        & \wt{S}  \ar[r]^{F_{\wt{S}}}   & \wt{S}
}
\]
$\wt{F}$ is said to be compatible with $\wt{D}$ if
$\wt{F}^*\OO_{\wt{X}_1}(-\wt{D}_1)=\OO_{\wt{X}}(-p\wt{D})$ holds.
Locally for the \'etale topology on $X$, there exists a lifting
$\wt{F}:\wt{X}\ra\wt{X}_1$ compatible with $\wt{D}$. Indeed, if $\wt{X}$ is
\'etale over $\A^n_{\wt{S}}$ via coordinates $\{\wt{x}_1,\cdots,\wt{x}_n\}$
and $\wt{D}=\divisor_0(\wt{x}_1\cdots\wt{x}_r)$, then there exists a unique
lifting $\wt{F}$ such that $\wt{F}^*(\wt{x}_i\otimes 1)=\wt{x}_i^p$ for
$1\leq i\leq n$.

We recall the following results from \cite[4.2.3]{di}. Two compatible
liftings $\wt{F}_1,\wt{F}_2$ differ by a derivation
\[ h_{12}=(\wt{F}_2^*-\wt{F}_1^*)/p:\Omega^1_{X_1/S}(\log D_1)\ra
F_*\OO_X. \]
In fact, if $\wt{X}$ is \'etale over $\A^n_{\wt{S}}$ via coordinates
$\{\wt{x}_1,\cdots,\wt{x}_n\}$ and $\wt{D}=\divisor_0(\wt{x}_1\cdots
\wt{x}_r)$, then we can write $\wt{F}_j^*(\wt{x}_i\otimes 1)=(1+pa_{ij})
\wt{x}_i^p$ for $1\leq i\leq r$ and $j=1,2$. By an easy calculation,
we have $h_{12}(dx_i/x_i\otimes 1)=a_{i2}-a_{i1}$ for $1\leq i\leq r$.
Furthermore, any lifting $\wt{F}$ compatible with $\wt{D}$ gives rise to
a quasi-isomorphism of complexes:
\begin{equation}
\phi_{\wt{F}}: \bigoplus_{i<p}\Omega^i_{X_1/S}(\log D_1)[-i]\ra
\tau_{<p}F_*\Omega^\bullet_{X/S}(\log D), \label{f2.1}
\end{equation}
which is given in degree 1 by $\phi^1_{\wt{F}}=\wt{F}^*/p$ and
prolonged canonically through the exterior powers.

\begin{thm}\label{2.2}
Let $f:X\ra Y$ be an $E$-semistable $S$-morphism with an adapted
divisor $D$, and $\wt{f}:(\wt{X},\wt{D})\ra (\wt{Y},\wt{E}_a)$ a lifting
of $f:(X,D)\ra (Y,E_a)$ over $\wt{S}$. Let $\wt{F}_{Y/S}:\wt{Y}\ra\wt{Y}_1$
be a lifting over $\wt{S}$ of the relative Frobenius morphism $F_{Y/S}:
Y\ra Y_1$ of $Y$ over $S$, which is compatible with the divisor $\wt{E}_a$.
Then there is a canonical isomorphism in $D(X',\OO_{X'})$:
\begin{equation}
\phi_{(\wt{f},\wt{F}_{Y/S})}:\bigoplus_{i<p}\Omega^i_{X'/Y}(\log D'/E_a)
[-i]\ra\tau_{<p}F_{X/Y*}\Omega^\bullet_{X/Y}(\log D/E_a), \label{f2.2}
\end{equation}
which induces the Cartier isomorphism $C^{-1}$ (\ref{f1.6}) on $\HH^i$.
\end{thm}

\begin{proof}
The proof is analogous to that of \cite[Theorem 2.2]{il90}. It suffices
to define, for any $i<p$, $\phi^i:\Omega^i_{X'/Y}(\log D'/E_a)[-i]
\ra F_{X/Y*}\Omega^\bullet_{X/Y}(\log D/E_a)$ inducing $C^{-1}$ on $\HH^i$.
Since $\phi^i$ can be deduced from $\phi^1$ by a similar argument to that
of \cite[2.1(a)]{di}, we have only to define $\phi^1$. The definition of
$\phi^1$ is given in three steps.

(Step 1: local case) To define $\phi^1$, we first suppose that there is
a lifting $\wt{F}:\wt{X}\ra\wt{X}_1$ of the relative Frobenius morphism
$F:X\ra X_1$ of $X$ over $S$, which is compatible with $\wt{D}$ and
compatible with $\wt{F}_{Y/S}$ in the sense that the square is commutative:
\begin{equation}
\xymatrix{
\wt{X} \ar[r]^{\wt{F}} \ar[d]_{\wt{f}} & \wt{X}_1 \ar[d]^{\wt{f}_1} \\
\wt{Y} \ar[r]^{\wt{F}_{Y/S}} & \wt{Y}_1.
} \label{f2.3}
\end{equation}
The morphism $\phi^1_{\wt{F}}:\Omega^1_{X_1/S}(\log D_1)[-1]\ra
F_*\Omega^\bullet_{X/S}(\log D)$ in (\ref{f2.1}), composed with the
projection of $F_*\Omega^\bullet_{X/S}(\log D)$ onto $F_*\Omega^
\bullet_{X/Y}(\log D/E_a)$ vanishes on the subsheaf
$f_1^*\Omega^1_{Y_1/S}(\log E_{a1})[-1]$, therefore by passage to
the quotient, it defines a morphism
$\Omega^1_{X_1/Y_1}(\log D_1/E_{a1})[-1]\ra F_*\Omega^\bullet_{X/Y}
(\log D/E_a)$, and by adjunction, it defines a morphism
\begin{equation}
\phi^1:\Omega^1_{X'/Y}(\log D'/E_a)[-1]\ra F_{X/Y*}\Omega^\bullet_{X/Y}
(\log D/E_a), \label{f2.4}
\end{equation}
which induces the Cartier isomorphism $C^{-1}$ on $\HH^1$.

(Step 2: from local to global) Assume that $\wt{F}_j:\wt{X}\ra\wt{X}_1$
are liftings of the relative Frobenius of $X$ over $S$ for $j=1,2$, which
are compatible with $\wt{D}$ and compatible with $\wt{F}_{Y/S}$. Then the
derivation $(\wt{F}_2^*-\wt{F}_1^*)/p:\Omega^1_{X_1/S}(\log D_1)\ra F_*\OO_X$
vanishes on the subsheaf $f^*_1\Omega^1_{Y_1/S}(\log E_{a1})$ by the
commutativity of the square (\ref{f2.3}). Therefore by passage to the quotient
and by adjunction, it defines a homomorphism
\[ h_{12}:\Omega^1_{X'/Y}(\log D'/E_a)\ra F_{X/Y*}\OO_X. \]
A calculation analogous to that of \cite[2.1(c)]{di}
shows that $\phi^1_2-\phi^1_1=dh_{12}$ holds, where $\phi^1_j$ are the
morphisms (\ref{f2.4}) associated to $\wt{F}_j$ for $j=1,2$.
By a similar argument to that of \cite[2.1(c)]{di}, we have a relation
of transitivity: $h_{12}+h_{23}=h_{13}$ for three liftings $\wt{F}_1,
\wt{F}_2,\wt{F}_3$ of the relative Frobenius of $X$ over $S$. Working on
the \'etale topology instead of the Zariski topology on $X$, we can
construct a global morphism $\phi^1$ by the procedure of the \v{C}ech
globalization described as in \cite[2.1(d)]{di}.

(Step 3: local existence of compatible liftings) We shall prove that locally for the
\'etale topology on $X$, there exists a lifting $\wt{F}:\wt{X}\ra\wt{X}_1$ compatible
with $\wt{D}$ and compatible with $\wt{F}_{Y/S}$. Keeping the types (i), (ii)
and (iii) as in the proof of Lemma \ref{1.5} in mind, we divide the argument
into four cases.

Case (I): assume $\divisor_0(y)\subset\wt{A}$ and $\wt{F}^*_{Y/S}(y\otimes 1)=(1+pa)y^p$.
Then we define $\wt{F}^*(x_1\otimes 1)=(1+pa)x_1^p$ and
$\wt{F}^*(x_i\otimes 1)=x_i^p$ ($i\geq 2$), where $x_1$ is the coordinate
for the fiber over $\divisor_0(y)$, and $x_i$ ($i\geq 2$) are the coordinates for
the divisor $\wt{D}_h$.

Case (II): assume $\divisor_0(y)\subset\wt{E}$ and $\wt{F}^*_{Y/S}(y\otimes 1)=(1+pa)y^p$.
Then we define $\wt{F}^*(x_1\otimes 1)=(1+pa)x_1^p$ and
$\wt{F}^*(x_i\otimes 1)=x_i^p$ ($i\geq 2$), where $x_1$ is a pre-chosen
coordinate for the fiber over $\divisor_0(y)$, and $x_i$ ($i\geq 2$) are the other
coordinates for the fiber over $\divisor_0(y)$ or the coordinates for $\wt{D}_h$.

Case (III): assume $\divisor_0(y)\not\subset\wt{E}_a$ and $\wt{F}^*_{Y/S}(y\otimes 1)=y^p+pb$.
Then we define $\wt{F}^*(x_1\otimes 1)=x_1^p+pb$ and
$\wt{F}^*(x_i\otimes 1)=x_i^p$ ($i\geq 2$), where $x_1$ is the coordinate
for the fiber over $\divisor_0(y)$, and $x_i$ ($i\geq 2$) are the coordinates for
$\wt{D}_h$.

Case (IV): assume that all $x_i$ are not the coordinates for the fiber over
$\divisor_0(y)$. Then we define $\wt{F}^*(x_i\otimes 1)=x_i^p$.

It is easy to check that $\wt{F}:\wt{X}\ra\wt{X}_1$ constructed above
is a lifting of the relative Frobenius of $X$ over $S$, which is compatible
with $\wt{D}$ and compatible with $\wt{F}_{Y/S}$.
\end{proof}

\begin{rem}\label{2.3}
(i) If $f$ is smooth, then the existence of a lifting of $(X',D'_h)$
over $\wt{Y}$ such that $\wt{X}'$ is smooth over $\wt{Y}$ and
$\wt{D}'_h$ is RSNC over $\wt{Y}$, gives rise to the decomposition
of $\tau_{<p}F_{X/Y*}\Omega^\bullet_{X/Y}(\log D_h)$.
Moreover, the gerbe of splittings of $\tau_{\leq 1}F_{X/Y*}\Omega^
\bullet_{X/Y}(\log D_h)$ is canonically isomorphic to the gerbe of
liftings of $(X',D'_h)$ over $\wt{Y}$ (see \cite[4.2.3]{di}).

(ii) Under the hypotheses of Theorem \ref{2.2}, suppose that $f$ is of
relative dimension $\leq p$ and $H^{p+1}(X',(\Omega^p_{X'/Y}(\log D'/E_a))
^\vee)=0$ (this is the case for example if $Y$ is affine and $f$ is proper),
then $F_{X/Y*}\Omega^\bullet_{X/Y}(\log D/E_a)$ is decomposable, i.e.\
there is an isomorphism in $D(X',\OO_{X'})$
\[ \bigoplus_i\Omega^i_{X'/Y}(\log D'/E_a)[-i]\stackrel{\sim}{\ra}F_{X/Y*}
\Omega^\bullet_{X/Y}(\log D/E_a), \]
which induces the Cartier isomorphism $C^{-1}$ on $\HH^i$. The proof of
the decomposition is analogous to that of \cite[3.7(b) and 4.2.3]{di}.
\end{rem}

We shall state some corollaries for $\Omega^\bullet_{X/Y}(\log D/E_a)$
and omit their proofs, which are analogous to those in \cite[\S 2]{il90}.

\begin{cor}\label{2.4}
Under the hypotheses of Theorem \ref{2.2}, suppose further that $f$
is proper. Then
\begin{itemize}
\item[(i)] For any $i+j<p$, the $\OO_Y$-modules $R^jf_*\Omega^i_{X/Y}
(\log D/E_a)$ are locally free of finite type, and of formation compatible
with any base change $Z\ra Y$.
\item[(ii)] The Hodge spectral sequence $E_1^{ij}=R^jf_*\Omega^i_{X/Y}
(\log D/E_a)\Rightarrow\bR^{i+j}f_*\Omega^\bullet_{X/Y}(\log D/E_a)$
satisfies $E_1^{ij}=E_\infty^{ij}$ for any $i+j<p$.
\item[(iii)] If $f$ is of relative dimension $\leq p$, then (i) and (ii)
are valid for any $i,j$.
\end{itemize}
\end{cor}

\begin{cor}\label{2.5}
Let $K$ be a field of characteristic zero, $S=\Spec K$, $X,Y$ smooth
$S$-schemes, and $f:X\ra Y$ a proper $E$-semistable $S$-morphism with an adapted
divisor $D$ as in Definition \ref{1.2}. Then
\begin{itemize}
\item[(i)] The $\OO_Y$-modules $R^jf_*\Omega^i_{X/Y}(\log D/E_a)$ are
locally free of finite type, and of formation compatible with any base
change $T\ra Y$.
\item[(ii)] The Hodge spectral sequence $E_1^{ij}=R^jf_*\Omega^i_{X/Y}
(\log D/E_a)\Rightarrow\bR^{i+j}f_*\Omega^\bullet_{X/Y}(\log D/E_a)$
degenerates in $E_1$.
\end{itemize}
\end{cor}

\begin{cor}\label{2.6}
Under the hypotheses of Corollary \ref{2.4}, suppose further that
$f$ is of purely relative dimension $d\leq p$ and $S$ is locally
noetherian and regular. Let $\LL$ be an $f$-ample invertible
$\OO_X$-module. Then we have
\begin{eqnarray}
R^jf_*(\LL^{-1}\otimes\Omega^i_{X/Y}(\log D/E_a))=0,
\quad \forall\,\, i+j<d, \nonumber \\
R^jf_*(\LL(-D_h)\otimes\Omega^i_{X/Y}(\log D/E_a))=0,
\quad \forall\,\, i+j>d. \nonumber
\end{eqnarray}
\end{cor}

\begin{cor}\label{2.7}
Under the hypotheses of Corollary \ref{2.5}, let $\LL$ be an $f$-ample
invertible $\OO_X$-module. Then we have
\begin{eqnarray}
R^jf_*(\LL^{-1}\otimes\Omega^i_{X/Y}(\log D/E_a))=0,
\quad \forall\,\, i+j<d, \nonumber \\
R^jf_*(\LL(-D_h)\otimes\Omega^i_{X/Y}(\log D/E_a))=0,
\quad \forall\,\, i+j>d. \nonumber
\end{eqnarray}
\end{cor}

\section{Variant with support}\label{S4}

In this section, let $S$ be a scheme of characteristic $p>0$, and
$f:X\ra Y$ an $E$-semistable $S$-morphism with an adapted divisor
$D$ as in Definition \ref{1.2}. For simplicity, denote $E_X+D_a$
by $D_v$. Tensoring (\ref{f1.1}) with $f^*\OO_Y(-E_a)=\OO_X(-D_v)$,
we obtain an exact sequence of locally free $\OO_X$-modules:
\begin{equation}
0\ra f^*\Omega^1_{Y/S}(E_a,0)\ra \Omega^1_{X/S}(D_v,D_h)\ra
\Omega^1_{X/Y}(D_v,D_h)\ra 0, \label{f4.1}
\end{equation}
where $\Omega^1_{Y/S}(E_a,0):=\Omega^1_{Y/S}(\log E_a)\otimes\OO_Y(-E_a)$,
$\Omega^1_{X/S}(D_v,D_h):=\Omega^1_{X/S}(\log D)\otimes\OO_X(-D_v)$, and
$\Omega^1_{X/Y}(D_v,D_h):=\Omega^1_{X/Y}(\log D/E_a)\otimes\OO_X(-D_v)$.
For any $i\geq 0$, define
\[ \Omega^i_{X/Y}(D_v,D_h)=\Omega^i_{X/Y}(\log D/E_a)
\otimes\OO_X(-D_v), \]
then it is easy to check that
$(\Omega^\bullet_{X/Y}(D_v,D_h),d)$ is a well-defined complex.

Let $F_Y$ be the absolute Frobenius of $Y$, and $F=F_{X/Y}$
the relative Frobenius of $X$ over $Y$. We have the following
commutative diagram with cartesian square:
\[
\xymatrix{
X \ar[dr]_f \ar[r]^F & X' \ar[d]^{f'} \ar[r] & X \ar[d]^{f} \\
                     & Y  \ar[r]^{F_Y}       & Y
}
\]

The differential $d$ of the complex $F_*\Omega^\bullet_{X/Y}(D_v,D_h)$
is $\OO_{X'}$-linear, so we would like to calculate its cohomology
$\OO_{X'}$-modules by a Cartier type isomorphism. Consider the Cartier
isomorphism (\ref{f1.6}) for $\Omega_{X/Y}^\bullet(\log D/E_a)$:
\[
C^{-1}:\Omega^*_{X'/Y}(\log D'/E_a)\ra\HH^*(F_*\Omega_{X/Y}^\bullet
(\log D/E_a)).
\]
Since $F_*\Omega_{X/Y}^\bullet(D_v,D_h)=F_*\Omega_{X/Y}^\bullet(\log D/E_a)
\otimes f^{\prime *}\OO_Y(-E_a)$, tensoring the above isomorphism with
$f^{\prime *}\OO_Y(-E_a)$, we have

\begin{prop}\label{4.1}
There is an isomorphism of graded $\OO_{X'}$-algebras:
\begin{equation}
C^{-1}:\Omega^*_{X'/Y}(D'_v,D'_h)\ra\HH^*(F_*\Omega^\bullet_{X/Y}
(D_v,D_h)), \label{f4.2}
\end{equation}
which is called the Cartier isomorphism of $\Omega^\bullet_{X/Y}(D_v,D_h)$.
\end{prop}

Tensoring (\ref{2.2}) with $f^{\prime *}\OO_Y(-E_a)$, we have

\begin{thm}\label{4.2}
Let $f:X\ra Y$ be an $E$-semistable $S$-morphism with an adapted
divisor $D$, and $\wt{f}:(\wt{X},\wt{D})\ra (\wt{Y},\wt{E}_a)$ a lifting
of $f:(X,D)\ra (Y,E_a)$ over $\wt{S}$. Let $\wt{F}_{Y/S}:\wt{Y}\ra\wt{Y}_1$
be a lifting over $\wt{S}$ of the relative Frobenius morphism $F_{Y/S}:
Y\ra Y_1$ of $Y$ over $S$, which is compatible with the divisor $\wt{E}_a$.
Then there is a canonical isomorphism in $D(X',\OO_{X'})$:
\begin{equation}
\phi_{(\wt{f},\wt{F}_{Y/S})}:\bigoplus_{i<p}\Omega^i_{X'/Y}(D'_v,D'_h)
[-i]\stackrel{\sim}{\ra}\tau_{<p}F_*\Omega^\bullet_{X/Y}(D_v,D_h),
\label{f4.3}
\end{equation}
which induces the Cartier isomorphism $C^{-1}$ (\ref{f4.2}) on $\HH^i$.
\end{thm}

\begin{rem}\label{4.6}
Under the hypotheses of Theorem \ref{4.2}, suppose that $f$ is of
relative dimension $\leq p$ and $H^{p+1}(X',(\Omega^p_{X'/Y}(D'_v,D'_h))
^\vee)=0$ (this is the case for example if $Y$ is affine and $f$ is proper),
then $F_*\Omega^\bullet_{X/Y}(D_v,D_h)$ is decomposable, i.e.\
there is an isomorphism in $D(X',\OO_{X'})$:
\[ \bigoplus_i\Omega^i_{X'/Y}(D'_v,D'_h)[-i]\stackrel{\sim}{\ra}F_*
\Omega^\bullet_{X/Y}(D_v,D_h), \]
which induces the Cartier isomorphism $C^{-1}$ on $\HH^i$. The proof of
the decomposition is analogous to that of \cite[3.7(b) and 4.2.3]{di}.
\end{rem}

We shall state some corollaries for $\Omega^\bullet_{X/Y}(D_v,D_h)$
and omit their proofs, which are analogous to those in \cite[\S 2]{il90}.

\begin{cor}\label{4.7}
Under the hypotheses of Theorem \ref{4.2}, suppose further that $f$
is proper. Then
\begin{itemize}
\item[(i)] For any $i+j<p$, the $\OO_Y$-modules $R^jf_*\Omega^i_{X/Y}
(D_v,D_h)$ are locally free of finite type, and of formation compatible
with any base change $Z\ra Y$.
\item[(ii)] The Hodge spectral sequence $E_1^{ij}=R^jf_*\Omega^i_{X/Y}
(D_v,D_h)\Rightarrow\bR^{i+j}f_*\Omega^\bullet_{X/Y}(D_v,D_h)$ satisfies
$E_1^{ij}=E_\infty^{ij}$ for any $i+j<p$.
\item[(iii)] If $f$ is of relative dimension $\leq p$, then (i) and (ii)
are valid for any $i,j$.
\end{itemize}
\end{cor}

\begin{cor}\label{4.8}
Let $K$ be a field of characteristic zero, $S=\Spec K$, $X,Y$ smooth
$S$-schemes, and $f:X\ra Y$ a proper $E$-semistable $S$-morphism with an
adapted divisor $D$ as in Definition \ref{1.2}. Then
\begin{itemize}
\item[(i)] The $\OO_Y$-modules $R^jf_*\Omega^i_{X/Y}(D_v,D_h)$ are
locally free of finite type, and of formation compatible with any base
change $T\ra Y$.
\item[(ii)] The Hodge spectral sequence $E_1^{ij}=R^jf_*\Omega^i_{X/Y}
(D_v,D_h)\Rightarrow\bR^{i+j}f_*\Omega^\bullet_{X/Y}(D_v,D_h)$
degenerates in $E_1$.
\end{itemize}
\end{cor}

\begin{cor}\label{4.9}
Under the hypotheses of Corollary \ref{4.7}, suppose further that
$f$ is of purely relative dimension $d\leq p$ and $S$ is locally
noetherian and regular. Let $\LL$ be an $f$-ample invertible
$\OO_X$-module. Then we have
\begin{eqnarray}
R^jf_*(\LL^{-1}\otimes\Omega^i_{X/Y}(D_v,D_h))=0,
\quad \forall\,\, i+j<d, \nonumber \\
R^jf_*(\LL(2D_v-D_h)\otimes\Omega^i_{X/Y}(D_v,D_h))=0,
\quad \forall\,\, i+j>d. \nonumber
\end{eqnarray}
\end{cor}

\begin{cor}\label{4.10}
Under the hypotheses of Corollary \ref{4.8}, let $\LL$ be an $f$-ample
invertible $\OO_X$-module. Then we have
\begin{eqnarray}
R^jf_*(\LL^{-1}\otimes\Omega^i_{X/Y}(D_v,D_h))=0,
\quad \forall\,\, i+j<d, \nonumber \\
R^jf_*(\LL(2D_v-D_h)\otimes\Omega^i_{X/Y}(D_v,D_h))=0,
\quad \forall\,\, i+j>d. \nonumber
\end{eqnarray}
\end{cor}

\section{Decomposition of de Rham complex with smooth horizontal
coefficients}\label{S5}

This section is parallel to \cite[\S 3]{il90}, and all proofs follow
Illusie's proofs very closely.

In this section, let $S$ be a scheme of characteristic $p>0$, and
$f:X\ra Y$ an $E$-semistable $S$-morphism with an adapted divisor
$D$ as in Definition \ref{1.2}. Then we have the following
exact sequence of locally free $\OO_X$-modules:
\begin{equation}
0\ra f^*\Omega^1_{Y/S}(\log E_a)\ra \Omega^1_{X/S}(\log D)\ra
\Omega^1_{X/Y}(\log D/E_a)\ra 0. \label{f5.1}
\end{equation}
By definition, $\Omega^i_{X/Y}(\log D/E_a)=\wedge^i\Omega^1_{X/Y}
(\log D/E_a)$ for any $i\geq 0$. Then we have
$\Omega^d_{X/Y}(\log D/E_a)=\omega_{X/Y}(D_h)$ and the de Rham
complex $(\Omega^\bullet_{X/Y}(\log D/E_a),d)$, where
\[
d:\Omega^i_{X/Y}(\log D/E_a)\ra\Omega^{i+1}_{X/Y}(\log D/E_a)
\]
is the ordinary differential map.

\begin{defn}\label{5.1}
Define $\bH=\oplus_i\bR^if_*\Omega^\bullet_{X/Y}(\log D/E_a)$ to be a
graded $\OO_Y$-module. The Koszul filtration of $\Omega^\bullet_{X/S}
(\log D)$ associated to (\ref{f5.1}) is defined as follows:
\[
K^i\Omega^\bullet_{X/S}(\log D)=\im(f^*\Omega^i_{Y/S}(\log E_a)\otimes
\Omega^{\bullet-i}_{X/S}(\log D)\ra\Omega^\bullet_{X/S}(\log D)).
\]
Then $K^i\Omega^\bullet_{X/S}(\log D)$ are subcomplexes of $\Omega^\bullet
_{X/S}(\log D)$ and induce a decreasing filtration of $\Omega^\bullet_{X/S}
(\log D)$:
\[
\cdots\supseteq K^i\Omega^\bullet_{X/S}(\log D)\supseteq
K^{i+1}\Omega^\bullet_{X/S}(\log D)\supseteq \cdots.
\]
It is easy to show that $K^i\Omega^j_{X/S}(\log D)$ is locally free
for any $i,j$ and the associated graded complex $\gr^i_K\Omega^\bullet
_{X/S}(\log D)=K^i/K^{i+1}=f^*\Omega^i_{Y/S}(\log E_a)\otimes\Omega^
{\bullet-i}_{X/Y}(\log D/E_a)$. The exact sequence
\[
0\ra K^1/K^2\ra K^0/K^2\ra K^0/K^1\ra 0
\]
is a short exact sequence of complexes:
\[
0\ra f^*\Omega^1_{Y/S}(\log E_a)\otimes\Omega^{\bullet-1}_{X/Y}(\log D/E_a)
\ra \Omega^\bullet_{X/S}(\log D)/K^2\ra
\Omega^{\bullet}_{X/Y}(\log D/E_a)\ra 0,
\]
which induces a morphism in $D(X)$:
\begin{equation}
\Omega^{\bullet}_{X/Y}(\log D/E_a)\ra f^*\Omega^1_{Y/S}(\log E_a)\otimes
\Omega^{\bullet}_{X/Y}(\log D/E_a). \label{f5.2}
\end{equation}

Applying $\oplus_i\bR^if_*$ to (\ref{f5.2}), we obtain the Gauss-Manin
connection
\begin{equation}
\nabla:\bH\ra\Omega^1_{Y/S}(\log E_a)\otimes\bH, \label{f5.3}
\end{equation}
and we can show that $\nabla$ is an integrable connection with logarithmic
poles along $E_a$. The complex
\begin{equation}
\Omega^\bullet_{Y/S}(\log E_a)(\bH)=(\bH\stackrel{\nabla}{\ra}
\Omega^1_{Y/S}(\log E_a)\otimes\bH\stackrel{\nabla}{\ra}\cdots
\stackrel{\nabla}{\ra}\Omega^i_{Y/S}(\log E_a)\otimes\bH
\stackrel{\nabla}{\ra}\cdots) \label{f5.4}
\end{equation}
is called the de Rham complex of $Y$ over $S$ with logarithmic poles along $E_a$
and coefficients in the Gauss-Manin system $\bH$. In fact, the Koszul filtration
of $\Omega^\bullet_{X/S}(\log D)$ and the derived functor $\bR f_*$ give rise to
a spectral sequence:
\begin{eqnarray}
E_1^{ij} & = & \bR^{i+j}f_*(\gr^i_K\Omega^\bullet_{X/S}(\log D)) \nonumber \\
& = & \Omega^i_{Y/S}(\log E_a)\otimes\bR^j f_*\Omega^\bullet_{X/Y}
(\log D/E_a)\Rightarrow\bR^{i+j} f_*\Omega^\bullet_{X/S}(\log D). \label{f5.5}
\end{eqnarray}
Then the de Rham complex $\Omega^\bullet_{Y/S}(\log E_a)(\bH)$ is just
the direct sum of the horizontal lines of $E_1^{ij}$ and the Gauss-Manin
connection $\nabla$ is just the direct sum of the differential operators
$d_1:E_1^{ij}\ra E_1^{i+1,j}$.
\end{defn}

\vskip 3mm
\noindent {\bf Variant.}
By definition, for any $i\geq 0$, $\Omega^i_{X/S}(D_v,D_h)=\Omega^i_{X/S}
(\log D)\otimes\OO_X(-D_v)$ and $\Omega^i_{X/Y}(D_v,D_h)=\Omega^i_{X/Y}
(\log D/E_a)\otimes\OO_X(-D_v)$.
Define $\bH^\dag=\oplus_i\bR^if_*\Omega^\bullet_{X/Y}(D_v,D_h)$ to be a
graded $\OO_Y$-module. The Koszul filtration of $\Omega^\bullet_{X/S}
(D_v,D_h)$ associated to (\ref{f5.1}) is defined as follows:
\[
K^i\Omega^\bullet_{X/S}(D_v,D_h)=\im(f^*\Omega^i_{Y/S}(\log E_a)\otimes
\Omega^{\bullet-i}_{X/S}(D_v,D_h)\ra\Omega^\bullet_{X/S}(D_v,D_h)).
\]
Then $K^i\Omega^\bullet_{X/S}(D_v,D_h)$ are subcomplexes of $\Omega^\bullet
_{X/S}(D_v,D_h)$ and induce a decreasing filtration of $\Omega^\bullet_{X/S}
(D_v,D_h)$:
\[
\cdots\supseteq K^i\Omega^\bullet_{X/S}(D_v,D_h)\supseteq
K^{i+1}\Omega^\bullet_{X/S}(D_v,D_h)\supseteq \cdots.
\]
It is easy to show that $K^i\Omega^j_{X/S}(D_v,D_h)$ is locally free
for any $i,j$ and the associated graded complex $\gr^i_K\Omega^\bullet
_{X/S}(D_v,D_h)=K^i/K^{i+1}=f^*\Omega^i_{Y/S}(\log E_a)\otimes\Omega^
{\bullet-i}_{X/Y}(D_v,D_h)$. The exact sequence
\[
0\ra K^1/K^2\ra K^0/K^2\ra K^0/K^1\ra 0
\]
is a short exact sequence of complexes:
\[
0\ra f^*\Omega^1_{Y/S}(\log E_a)\otimes\Omega^{\bullet-1}_{X/Y}(D_v,D_h)
\ra \Omega^\bullet_{X/S}(D_v,D_h)/K^2\ra
\Omega^{\bullet}_{X/Y}(D_v,D_h)\ra 0,
\]
which induces a morphism in $D(X)$:
\begin{equation}
\Omega^{\bullet}_{X/Y}(D_v,D_h)\ra f^*\Omega^1_{Y/S}(\log E_a)\otimes
\Omega^{\bullet}_{X/Y}(D_v,D_h). \label{f5.6}
\end{equation}

Applying $\oplus_i\bR^if_*$ to (\ref{f5.6}), we obtain the Gauss-Manin
connection
\begin{equation}
\nabla:\bH^\dag\ra\Omega^1_{Y/S}(\log E_a)\otimes\bH^\dag, \label{f5.7}
\end{equation}
and we can show that $\nabla$ is an integrable connection with logarithmic
poles along $E_a$. The complex
\begin{equation}
\Omega^\bullet_{Y/S}(\log E_a)(\bH^\dag)=(\bH^\dag\stackrel{\nabla}{\ra}
\Omega^1_{Y/S}(\log E_a)\otimes\bH^\dag\stackrel{\nabla}{\ra}\cdots
\stackrel{\nabla}{\ra}\Omega^i_{Y/S}(\log E_a)\otimes\bH^\dag
\stackrel{\nabla}{\ra}\cdots) \label{f5.8}
\end{equation}
is called the de Rham complex of $Y$ over $S$ with logarithmic poles along $E_a$
and coefficients in the Gauss-Manin system $\bH^\dag$. In fact, the Koszul filtration
of $\Omega^\bullet_{X/S}(D_v,D_h)$ and the derived functor $\bR f_*$
give rise to a spectral sequence:
\begin{eqnarray}
E_1^{ij} & = & \bR^{i+j}f_*(\gr^i_K\Omega^\bullet_{X/S}(D_v,D_h)) \nonumber \\
& = & \Omega^i_{Y/S}(\log E_a)\otimes\bR^j f_*\Omega^\bullet_{X/Y}
(D_v,D_h)\Rightarrow\bR^{i+j} f_*\Omega^\bullet_{X/S}(D_v,D_h). \label{f5.9}
\end{eqnarray}
Then the de Rham complex $\Omega^\bullet_{Y/S}(\log E_a)(\bH^\dag)$ is just
the direct sum of the horizontal lines of $E_1^{ij}$ and the Gauss-Manin
connection $\nabla$ is just the direct sum of the differential operators
$d_1:E_1^{ij}\ra E_1^{i+1,j}$.

\begin{defn}\label{5.2}
The Hodge filtrations of $\bH$ and $\bH^\dag$ are decreasing
filtrations defined respectively by:
\begin{eqnarray*}
\Fil^i\bH &=& \im(\oplus_j\bR^jf_*\Omega^{\geq i}_{X/Y}(\log D/E_a)
\ra\oplus_j\bR^jf_*\Omega^\bullet_{X/Y}(\log D/E_a)), \\
\Fil^i\bH^\dag &=& \im(\oplus_j\bR^jf_*\Omega^{\geq i}_{X/Y}(D_v,D_h)
\ra\oplus_j\bR^jf_*\Omega^\bullet_{X/Y}(D_v,D_h)),
\end{eqnarray*}
which induce the Hodge spectral sequences:
\begin{eqnarray}
E_1^{ij} &=& R^jf_*\Omega^i_{X/Y}(\log D/E_a)\Rightarrow
\bR^{i+j}f_*\Omega^\bullet_{X/Y}(\log D/E_a), \label{f5.14} \\
E_1^{ij} &=& R^jf_*\Omega^i_{X/Y}(D_v,D_h)\Rightarrow
\bR^{i+j}f_*\Omega^\bullet_{X/Y}(D_v,D_h). \label{f5.15}
\end{eqnarray}

Note that the Gauss-Manin connection satisfies Griffiths' transverality:
\[
\nabla(\Fil^i\bH^!)\subset\Omega^1_{Y/S}(\log E_a)\otimes\Fil^{i-1}\bH^!,
\,\,\hbox{where}\,\,!=\,\,\hbox{null or}\,\,\dag,
\]
hence the Hodge filtration of $\bH^!$ induces a decreasing filtration
of the de Rham complex $\Omega^\bullet_{Y/S}(\log E_a)(\bH^!)$ by
subcomplexes:
\begin{eqnarray}
\Fil^i\Omega^\bullet_{Y/S}(\log E_a)(\bH^!)=
& (\Fil^i\bH^! & \stackrel{\nabla}{\ra} \Omega^1_{Y/S}(\log E_a)
\otimes\Fil^{i-1}\bH^!\stackrel{\nabla}{\ra} \nonumber \\
& \cdots & \stackrel{\nabla}{\ra}
\Omega^j_{Y/S}(\log E_a)\otimes\Fil^{i-j}\bH^!
\stackrel{\nabla}{\ra}\cdots). \label{f5.17}
\end{eqnarray}
Assume that the Hodge spectral sequences (\ref{f5.14}) and (\ref{f5.15})
degenerate in $E_1$. Then we have
\begin{eqnarray}
\oplus_j R^{j-i}f_*\Omega^i_{X/Y}(\log D/E_a) & \stackrel{\sim}{\ra} &
\gr^i\bH, \label{f5.18} \\
\oplus_j R^{j-i}f_*\Omega^i_{X/Y}(D_v,D_h) & \stackrel{\sim}{\ra} &
\gr^i\bH^\dag. \label{f5.19}
\end{eqnarray}
A similar argument to \cite{katz} shows that the Gauss-Manin connection
$\nabla:\gr^i\bH^!\ra\Omega^1_{Y/S}(\log E_a)\otimes\gr^{i-1}\bH^!$ can
be identified with the cup product by the Kodaira-Spencer class
$c\in\Ext^1_{\OO_X}(\Omega^1_{X/Y}(\log D/E_a),f^*\Omega^1_{Y/S}(\log E_a))$
defined by (\ref{f5.1}). For this reason, the graded complex of
$\Omega^\bullet_{Y/S}(\log E_a)(\bH^!)$ associated to the Hodge
filtration (\ref{f5.17}) is called the Kodaira-Spencer complex:
\begin{eqnarray}
\gr^i\Omega^\bullet_{Y/S}(\log E_a)(\bH^!)=
& (\gr^i\bH^! & \stackrel{\nabla}{\ra} \Omega^1_{Y/S}(\log E_a)
\otimes\gr^{i-1}\bH^!\stackrel{\nabla}{\ra} \nonumber \\
& \cdots & \stackrel{\nabla}{\ra}
\Omega^j_{Y/S}(\log E_a)\otimes\gr^{i-j}\bH^!
\stackrel{\nabla}{\ra}\cdots), \label{f5.21}
\end{eqnarray}
where $!=$ null or $\dag$.
\end{defn}

\begin{defn}\label{5.3}
The conjugate filtrations of $\bH$ and $\bH^\dag$ are increasing
filtrations defined respectively by:
\begin{eqnarray*}
\Fil_i\bH &=& \im(\oplus_j\bR^jf_*(\tau_{\leq i}\Omega^\bullet_{X/Y}
(\log D/E_a))\ra\oplus_j\bR^jf_*\Omega^\bullet_{X/Y}(\log D/E_a)), \\
\Fil_i\bH^\dag &=& \im(\oplus_j\bR^jf_*(\tau_{\leq i}\Omega^\bullet_{X/Y}
(D_v,D_h))\ra\oplus_j\bR^jf_*\Omega^\bullet_{X/Y}(D_v,D_h)),
\end{eqnarray*}
which induce the conjugate spectral sequences:
\begin{eqnarray}
_cE_2^{ij} &=& R^if_*\HH^j(\Omega^\bullet_{X/Y}(\log D/E_a))\Rightarrow
\bR^{i+j}f_*\Omega^\bullet_{X/Y}(\log D/E_a), \label{f5.22} \\
_cE_2^{ij} &=& R^if_*\HH^j(\Omega^\bullet_{X/Y}(D_v,D_h))\Rightarrow
\bR^{i+j}f_*\Omega^\bullet_{X/Y}(D_v,D_h). \label{f5.23}
\end{eqnarray}

The conjugate filtration is stable under the Gauss-Manin connection, i.e.\
\[
\nabla(\Fil_i\bH^!)\subseteq\Omega^1_{Y/S}(\log E_a)\otimes\Fil_i\bH^!,
\]
hence the conjugate filtration of $\bH^!$ induces an increasing filtration
of the de Rham complex $\Omega^\bullet_{Y/S}(\log E_a)(\bH^!)$ by
subcomplexes:
\begin{eqnarray}
\Fil_i\Omega^\bullet_{Y/S}(\log E_a)(\bH^!)=
& (\Fil_i\bH^! & \stackrel{\nabla}{\ra} \Omega^1_{Y/S}(\log E_a)
\otimes\Fil_i\bH^!\stackrel{\nabla}{\ra} \nonumber \\
& \cdots & \stackrel{\nabla}{\ra}
\Omega^j_{Y/S}(\log E_a)\otimes\Fil_i\bH^!
\stackrel{\nabla}{\ra}\cdots). \label{f5.25}
\end{eqnarray}
From the increasing filtration $\Fil_i$ of $\Omega^\bullet_{Y/S}
(\log E_a)(\bH^!)$, we obtain a decreasing filtration $\Fil_{-i}$
of $\Omega^\bullet_{Y/S}(\log E_a)(\bH^!)$, which gives rise to
a spectral sequence:
\begin{equation}
E_1^{ij}=\HH^{i+j}(\gr_{-i}\Omega^\bullet_{Y/S}(\log E_a)(\bH^!))
\Rightarrow\HH^{i+j}(\Omega^\bullet_{Y/S}(\log E_a)(\bH^!)), \label{f5.26}
\end{equation}
where $!=$ null or $\dag$.
\end{defn}

From now on, let $\wt{S}$ be a lifting of $S$ over $\Z/p^2\Z$, and
$F_{\wt{S}}:\wt{S}\ra\wt{S}$ a lifting of the absolute Frobenius
morphism $F_S:S\ra S$ over $\wt{S}$. We need the following assumptions:
\begin{asp}\label{5.4}
\begin{itemize}
\item[(i)] $f:X\ra Y$ is proper and of relative dimension $\leq p$;
\item[(ii)] $f:(X,D)\ra(Y,E_a)$ has a lifting $\wt{f}:(\wt{X},\wt{D})
\ra(\wt{Y},\wt{E}_a)$ over $\wt{S}$; and
\item[(iii)] $F_{Y/S}:Y\ra Y_1$ has a lifting $\wt{F}_{Y/S}:
\wt{Y}\ra\wt{Y}_1$ over $\wt{S}$, which is compatible with $\wt{E}_a$.
\end{itemize}
\end{asp}

Under Assumption \ref{5.4}, by Corollaries \ref{2.4} and \ref{4.7},
we have that for any $i,j$, $R^jf_*\Omega^i_{X/Y}(\log D/E_a)$ and
$R^jf_*\Omega^i_{X/Y}(D_v,D_h)$ are locally free of finite type,
and of formation compatible with any base change, and that the Hodge
spectral sequences (\ref{f5.14}) and (\ref{f5.15}) degenerate in $E_1$.
Furthermore, we have

\begin{lem}\label{5.5}
Under Assumption \ref{5.4}, the conjugate spectral sequences
(\ref{f5.22}) and (\ref{f5.23}) degenerate in $E_2$.
\end{lem}

\begin{proof}
It is a direct consequence of the Cartier isomorphism.
For the degeneracy of (\ref{f5.22}), we use Proposition \ref{1.7}
and the degeneracy of (\ref{f5.14}).
For the degeneracy of (\ref{f5.23}), we use Proposition \ref{4.1}
and the degeneracy of (\ref{f5.15}).
\end{proof}

For convenience of the reader, we recall the following commutative
diagram with cartesian squares:
\begin{equation}
\xymatrix{
X \ar[dr]_f \ar[r]^{F=F_{X/Y}} & X' \ar[d]^{f'} \ar[r] &
X_1 \ar[d]^{f_1} \ar[r] & X \ar[d]^{f} \\
 & Y \ar[r]^{F_{Y/S}} \ar[dr] & Y_1 \ar[r] \ar[d] & Y \ar[d] \\
 & & S \ar[r]^{F_S} & S.
} \label{f5.27}
\end{equation}

In the remaining part of this section, we always assume $!=$ null or $\dag$
for $\bH^!$ unless otherwise stated. The degeneracy of the conjugate spectral
sequences (\ref{f5.22}) and (\ref{f5.23}) in $E_2$ gives rise to the following
isomorphisms:
\begin{eqnarray}
&\gr_i\bH& \stackrel{\sim}{\ra}\oplus_jR^{j-i}f_*\HH^i(\Omega^\bullet_{X/Y}
(\log D/E_a))\stackrel{\sim}{\ra}\oplus_jR^{j-i}f'_*\HH^i(F_*\Omega^\bullet
_{X/Y}(\log D/E_a)), \nonumber \\
&\gr_i\bH^\dag& \stackrel{\sim}{\ra}\oplus_jR^{j-i}f_*\HH^i(\Omega^\bullet
_{X/Y}(D_v,D_h))\stackrel{\sim}{\ra}\oplus_jR^{j-i}f'_*\HH^i(F_*\Omega^\bullet
_{X/Y}(D_v,D_h)). \nonumber
\end{eqnarray}
By the Cartier isomorphisms and the base changes in (\ref{f5.27}),
we have
\begin{eqnarray}
\gr_i\bH &\stackrel{\sim}{\ra}& \oplus_jR^{j-i}f'_*\Omega^i_{X'/Y}
(\log D'/E_a)\stackrel{\sim}{\ra} \oplus_jF^*_{Y/S}R^{j-i}f_{1*}
\Omega^i_{X_1/Y_1}(\log D_1/E_{a1}) \nonumber \\
&=& F^*_{Y/S}\gr^i\bH_1=\gr^i\bH_1\otimes\OO_Y, \label{f5.28} \\
\gr_i\bH^\dag &\stackrel{\sim}{\ra}& \oplus_jR^{j-i}f'_*\Omega^i_{X'/Y}
(D'_v,D'_h)\stackrel{\sim}{\ra} \oplus_jF^*_{Y/S}R^{j-i}f_{1*}
\Omega^i_{X_1/Y_1}(D_{v1},D_{h1}) \nonumber \\
&=& F^*_{Y/S}\gr^i\bH^\dag_1=\gr^i\bH^\dag_1\otimes\OO_Y.
\label{f5.29}
\end{eqnarray}
Furthermore, \cite[2.3.1.3]{katz} shows that the Gauss-Manin connection
satisfies $\nabla_{\gr_i}=1\otimes d$ under the above isomorphisms,
hence we obtain the following isomorphism of complexes, where the left
one is the graded complex associated to (\ref{f5.25}),
and the differential of the right one is $1\otimes d$:
\begin{equation}
F_{Y/S*}\gr_i\Omega^\bullet_{Y/S}(\log E_a)(\bH^!)\stackrel{\sim}{\ra}
\gr^i\bH^!_1\otimes F_{Y/S*}\Omega^\bullet_{Y/S}(\log E_a). \label{f5.30}
\end{equation}
Since $\gr^i\bH^!_1$ is locally free, we have the isomorphism for $E_1$
terms in (\ref{f5.26}):
\begin{eqnarray}
E_1^{-i+j,i}(F_{Y/S*}\Omega^\bullet_{Y/S}(\log E_a)(\bH^!),\Fil_\bullet)
=\HH^j(F_{Y/S*}\gr_{i-j}\Omega^\bullet_{Y/S}(\log E_a)(\bH^!)) \nonumber \\
\stackrel{\sim}{\ra}\gr^{i-j}\bH^!_1\otimes\HH^j(F_{Y/S*}\Omega^
\bullet_{Y/S}(\log E_a))\stackrel{\stackrel{C}{\sim}}{\ra}
\gr^{i-j}\bH^!_1\otimes\Omega^j_{Y_1/S}(\log E_{a1}), \nonumber
\end{eqnarray}
whose inverse is called the Cartier isomorphism for
$\Omega^\bullet_{Y/S}(\log E_a)(\bH^!)$:
\begin{equation}
C^{-1}:\gr^{i-j}\bH^!_1\otimes\Omega^j_{Y_1/S}(\log E_{a1})
\stackrel{\sim}{\ra}
E_1^{-i+j,i}(F_{Y/S*}\Omega^\bullet_{Y/S}(\log E_a)(\bH^!),\Fil_\bullet).
\label{f5.31}
\end{equation}
Note that the left hand side of (\ref{f5.31}) is the $j$-term in
the Kodaira-Spencer complex (\ref{f5.21}) of $\bH^!_1$ on $Y_1$.
It follows from \cite[3.2]{katz} that the right hand side of
(\ref{f5.31}) with the differential $d_1$ up to sign corresponds
to the Kodaira-Spencer complex.

By definition \cite[1.3.3]{de}, the delayed filtration $G_\bullet
=Dec(\Fil_\bullet)$ associated to the conjugate filtration $\Fil_
\bullet$ is an increasing filtration of the complex $\Omega^\bullet
_{Y/S}(\log E_a)(\bH^!)$, which is defined by
\[
G_i\Omega^j_{Y/S}(\log E_a)(\bH^!)=\{x\in\Omega^j_{Y/S}(\log E_a)
\otimes\Fil_{i-j}\bH^!\,|\,\nabla(x)\in\Omega^{j+1}_{Y/S}(\log E_a)
\otimes\Fil_{i-j-1}\bH^! \}.
\]
Similarly, we have also an increasing filtration of $F_{Y/S*}\Omega^
\bullet_{Y/S}(\log E_a)(\bH^!)$ by subcomplexes of $\OO_{Y_1}$-modules.
There is a natural surjective homomorphism
\[
\gr^G_i\Omega^j_{Y/S}(\log E_a)(\bH^!)\ra E_1^{-i+j,i}(\Omega^\bullet_{Y/S}
(\log E_a)(\bH^!),\Fil_\bullet),
\]
which is indeed an isomorphism and induces isomorphisms for all $r\geq 1$:
\[
E_r(\Omega^\bullet_{Y/S}(\log E_a)(\bH^!),G_\bullet)\stackrel{\sim}{\ra}
E_{r+1}(\Omega^\bullet_{Y/S}(\log E_a)(\bH^!),\Fil_\bullet).
\]

\vskip 3mm
\noindent{\bf Objective.} Under Assumption \ref{5.4}, we shall construct
a decomposition in $D(Y_1)$:
\[
G_{p-1}F_{Y/S*}\Omega^\bullet_{Y/S}(\log E_a)(\bH^!)\stackrel{\sim}{\ra}
\bigoplus_{i<p}\gr^i\Omega^\bullet_{Y_1/S}(\log E_{a1})(\bH^!_1).
\]

Fix $i<p$. For any $j\geq 0$, the decompositions in Theorems \ref{2.2}
and \ref{4.2} give rise to the morphisms in $D(X')$:
\begin{eqnarray}
&\phi^{i-j}_{(\wt{f},\wt{F}_{Y/S})}:&\Omega^{i-j}_{X'/Y}(\log D'/E_a)[-i+j]
\ra\tau_{\leq i-j}F_*\Omega^\bullet_{X/Y}(\log D/E_a), \label{f5.32} \\
&\phi^{i-j}_{(\wt{f},\wt{F}_{Y/S})}:&\Omega^{i-j}_{X'/Y}(D'_v,D'_h)[-i+j]
\ra\tau_{\leq i-j}F_*\Omega^\bullet_{X/Y}(D_v,D_h). \label{f5.33}
\end{eqnarray}
Applying $\oplus_k\bR^kf'_*$ to (\ref{f5.32}) and (\ref{f5.33}),
we obtain the homomorphisms of $\OO_Y$-modules:
\begin{eqnarray}
u^{i-j}:\gr^{i-j}\bH' &=& \oplus_kR^{k-i+j}f'_*\Omega^{i-j}_{X'/Y}
(\log D'/E_a) \nonumber \\
&\ra& \Fil_{i-j}(\oplus_k\bR^kf_*\Omega^\bullet_{X/Y}(\log D/E_a))
=\Fil_{i-j}\bH, \label{f5.34} \\
u^{i-j}:\gr^{i-j}\bH^{\dag\prime} &=& \oplus_kR^{k-i+j}f'_*\Omega
^{i-j}_{X'/Y}(D'_v,D'_h) \nonumber \\
&\ra& \Fil_{i-j}(\oplus_k\bR^kf_*\Omega^\bullet_{X/Y}(D_v,D_h))
=\Fil_{i-j}\bH^\dag. \label{f5.35}
\end{eqnarray}
On the other hand, $\wt{F}_{Y/S}$ gives rise to the homomorphism:
\[
v^1=\wt{F}^*_{Y/S}/p: \Omega^1_{Y_1/S}(\log E_{a1})\ra\ZZ^1(F_{Y/S*}
\Omega^\bullet_{Y/S}(\log E_a)),
\]
which, by exterior product, induces the homomorphism:
\begin{equation}
v^j:\Omega^j_{Y_1/S}(\log E_{a1})\ra\ZZ^j(F_{Y/S*}\Omega^\bullet_{Y/S}
(\log E_a))\subset F_{Y/S*}\Omega^j_{Y/S}(\log E_a). \label{f5.36}
\end{equation}
By adjunction of $(F^*_{Y/S},F_{Y/S*})$ and abuse of notation,
(\ref{f5.34}) and (\ref{f5.35}) yield the homomorphism:
\begin{equation}
u^{i-j}:\gr^{i-j}\bH^!_1\ra F_{Y/S*}\Fil_{i-j}\bH^!. \label{f5.37}
\end{equation}
Combining (\ref{f5.36}) and (\ref{f5.37}), we obtain the
homomorphism of $\OO_{Y_1}$-modules:
\begin{equation}
v^j\otimes u^{i-j}:\Omega^j_{Y_1/S}(\log E_{a1})\otimes\gr^{i-j}
\bH^!_1\ra F_{Y/S*}(\Omega^j_{Y/S}(\log E_a)\otimes\Fil_{i-j}\bH^!).
\label{f5.38}
\end{equation}

\begin{prop}\label{5.6}
Under Assumption \ref{5.4}, we have
\begin{itemize}
\item[(i)] The image of $v^j\otimes u^{i-j}$ is contained in
$G_iF_{Y/S*}\Omega^j_{Y/S}(\log E_a)(\bH^!)$, where $G_\bullet$
is the delayed filtration.
\item[(ii)] For any $i<p$ and any $j\geq 0$, the following square is
commutative:
\begin{equation}
\xymatrix{
\Omega^j_{Y_1/S}(\log E_{a1})\otimes\gr^{i-j}\bH^!_1
\ar[d]_{v^j\otimes u^{i-j}} \ar[r]^\nabla &
\Omega^{j+1}_{Y_1/S}(\log E_{a1})\otimes\gr^{i-j-1}\bH^!_1
\ar[d]^{v^{j+1}\otimes u^{i-j-1}} \\
F_{Y/S*}(\Omega^j_{Y/S}(\log E_a)(\bH^!)) \ar[r]^{F_{Y/S*}\nabla} &
F_{Y/S*}(\Omega^{j+1}_{Y/S}(\log E_a)(\bH^!)),
} \label{f5.39}
\end{equation}
where the upper horizontal morphism is the differential map of the
Kodaira-Spencer complex. (i) and (ii) give rise to a morphism of complexes:
\begin{equation}
(v\otimes u)^i:\gr^i\Omega^\bullet_{Y_1/S}(\log E_{a1})(\bH^!_1)
\ra G_iF_{Y/S*}\Omega^\bullet_{Y/S}(\log E_a)(\bH^!). \label{f5.40}
\end{equation}
\item[(iii)] The composition of the morphisms of complexes
\begin{eqnarray}
\gr^i\Omega^\bullet_{Y_1/S}(\log E_{a1})(\bH^!_1)\stackrel{(\ref{f5.40})}
{\ra}G_iF_{Y/S*}\Omega^\bullet_{Y/S}(\log E_a)(\bH^!)\ra \nonumber \\
\gr^G_iF_{Y/S*}\Omega^\bullet_{Y/S}(\log E_a)(\bH^!)\ra E_1^
{-i+\bullet,i}(F_{Y/S*}\Omega^\bullet_{Y/S}(\log E_a)(\bH^!),\Fil_\bullet)
\nonumber
\end{eqnarray}
induces the Cartier isomorphism (\ref{f5.31}) for $\Omega^\bullet_{Y/S}
(\log E_a)(\bH^!)$, hence is a quasi-isomorphism. Consequently,
the following morphism is a quasi-isomorphism:
\begin{equation}
\sum_{i<p}(v\otimes u)^i:\bigoplus_{i<p}\gr^i\Omega^\bullet_{Y_1/S}
(\log E_{a1})(\bH^!_1)\ra G_{p-1}F_{Y/S*}\Omega^\bullet_{Y/S}(\log E_a)
(\bH^!). \label{f5.41}
\end{equation}
\end{itemize}
\end{prop}

The essential point in the proof of Proposition \ref{5.6} is the
compatibility (ii), which is deduced from a more general compatibility
in the level of derived category between the morphism $u$ and
the Gauss-Manin connection $\nabla$.

\begin{lem}\label{5.7}
The Koszul filtrations $K^\bullet$ of the complexes $\Omega^\bullet_{X/S}
(\log D)$ and $\Omega^\bullet_{X/S}(D_v,D_h)$ give rise to the short
exact sequences of complexes:
\begin{eqnarray*}
0\ra \gr^{i+1}_K\Omega^\bullet_{X/S}(\log D)\ra K^i/K^{i+2}(\Omega
^\bullet_{X/S}(\log D))\ra\gr^i_K\Omega^\bullet_{X/S}(\log D)\ra 0, \\
0\ra \gr^{i+1}_K\Omega^\bullet_{X/S}(D_v,D_h)\ra K^i/K^{i+2}(\Omega
^\bullet_{X/S}(D_v,D_h))\ra\gr^i_K\Omega^\bullet_{X/S}(D_v,D_h)\ra 0,
\end{eqnarray*}
which yield the connecting morphisms $\partial:\Gamma(i)\ra\Gamma'(i)$
in $D(X)$, where
\[
\Gamma(i)=\left\{
\begin{array}{ll}
f^*\Omega^i_{Y/S}(\log E_a)\otimes\Omega^{\bullet-i}_{X/Y}(\log D/E_a) &
\hbox{for}\,\, \Omega^\bullet_{X/S}(\log D) \\
f^*\Omega^i_{Y/S}(\log E_a)\otimes\Omega^{\bullet-i}_{X/Y}(D_v,D_h) &
\hbox{for}\,\, \Omega^\bullet_{X/S}(D_v,D_h),
\end{array} \right.
\]
\[
\Gamma'(i)=\left\{
\begin{array}{ll}
f^*\Omega^{i+1}_{Y/S}(\log E_a)\otimes\Omega^{\bullet-i}_{X/Y}(\log D/E_a) &
\hbox{for}\,\, \Omega^\bullet_{X/S}(\log D) \\
f^*\Omega^{i+1}_{Y/S}(\log E_a)\otimes\Omega^{\bullet-i}_{X/Y}(D_v,D_h) &
\hbox{for}\,\, \Omega^\bullet_{X/S}(D_v,D_h).
\end{array} \right.
\]
Then for any $i,j$, the following square is commutative:
\begin{equation}
\xymatrix{
\Gamma(i)\otimes\Gamma(j) \ar[r] \ar[d]^\pi & (\Gamma'(i)\otimes\Gamma(j))
\oplus(\Gamma(i)\otimes\Gamma'(j)) \ar[d]^{\pi+\pi} \\
\Gamma(i+j) \ar[r]^\partial & \Gamma'(i+j),
} \label{f5.42}
\end{equation}
where the upper horizontal morphism is $\partial\otimes 1+1\otimes\partial$,
and $\pi$ is the product morphism composed possibly with an isomorphism
of commutativity.
\end{lem}

\begin{proof}
It suffices to prove that the product morphism $\pi:\Omega^\bullet
_{X/S}(\log D)\otimes\Omega^\bullet_{X/S}(\log D)\ra\Omega^\bullet
_{X/S}(\log D)$ is compatible with the Koszul filtration. Thus we
can use the morphisms of the corresponding short exact sequences of
$K^n/K^{n+2}(\Omega^\bullet_{X/S}(\log D)\otimes\Omega^\bullet_{X/S}
(\log D))\ra K^n/K^{n+2}(\Omega^\bullet_{X/S}(\log D))$ to obtain the
conclusion. The proof for $\Omega^\bullet_{X/S}(D_v,D_h)$ is similar.
\end{proof}

Applying $\oplus_k\bR^kf_*$ to (\ref{f5.42}), we obtain the following
commutative square:
\begin{equation}
\xymatrix{
\Theta(i)\otimes\Theta(j) \ar[r] \ar[d]^\pi & (\Theta(i+1)\otimes\Theta(j))
\oplus(\Theta(i)\otimes\Theta(j+1)) \ar[d]^{\pi+\pi} \\
\Theta(i+j) \ar[r]^\nabla & \Theta(i+j+1),
} \label{f5.43}
\end{equation}
where $\Theta(i)=\Omega^i_{Y/S}(\log E_a)(\bH^!)$, the upper horizontal
morphism is $\nabla\otimes 1+1\otimes\nabla$, and $\pi$ is the product
morphism composed possibly with an isomorphism of commutativity.
The diagram (\ref{f5.43}) implies that the complex
$\Omega^\bullet_{Y/S}(\log E_a)(\bH^!)$ is a differential
graded module over $\Omega^\bullet_{Y/S}(\log E_a)$.

\begin{lem}\label{5.8}
For any $i<p$ and any $j\geq 0$, the following squares are commutative:
\[\scriptsize
\xymatrix{
f^*_1\Omega^j_{Y_1/S}(\log E_{a1})\otimes\Omega^{i-j}_{X_1/Y_1}
(\log D_1/E_{a1})[-i+j] \ar[r]^\partial \ar[d]_{v^j\otimes\phi^{i-j}} &
f^*_1\Omega^{j+1}_{Y_1/S}(\log E_{a1})\otimes\Omega^{i-j-1}_{X_1/Y_1}
(\log D_1/E_{a1})[-i+j+1] \ar[d]^{v^{j+1}\otimes\phi^{i-j-1}} \\
F_{X/S*}(f^*\Omega^j_{Y/S}(\log E_a)\otimes\Omega^\bullet_{X/Y}
(\log D/E_a)) \ar[r]^{F_{X/S*}\partial} &
F_{X/S*}(f^*\Omega^{j+1}_{Y/S}(\log E_a)\otimes\Omega^\bullet_{X/Y}
(\log D/E_a)),
}
\]
\[\scriptsize
\xymatrix{
f^*_1\Omega^j_{Y_1/S}(\log E_{a1})\otimes\Omega^{i-j}_{X_1/Y_1}
(D_{v1},D_{h1})[-i+j] \ar[r]^\partial \ar[d]_{v^j\otimes\phi^{i-j}} &
f^*_1\Omega^{j+1}_{Y_1/S}(\log E_{a1})\otimes\Omega^{i-j-1}_{X_1/Y_1}
(D_{v1},D_{h1})[-i+j+1] \ar[d]^{v^{j+1}\otimes\phi^{i-j-1}} \\
F_{X/S*}(f^*\Omega^j_{Y/S}(\log E_a)\otimes\Omega^\bullet_{X/Y}
(D_v,D_h)) \ar[r]^{F_{X/S*}\partial} &
F_{X/S*}(f^*\Omega^{j+1}_{Y/S}(\log E_a)\otimes\Omega^\bullet_{X/Y}
(D_v,D_h)),
}
\]
where $\phi^{i-j}$ are deduced from (\ref{f5.32}) and (\ref{f5.33})
by adjunction of $(F^*_{Y/S},F_{Y/S*})$, and the upper and lower
horizontal morphisms are deduced from the short exact sequences of
$K^j/K^{j+2}(\Omega^\bullet_{X/S}(\log D))$ and
$K^j/K^{j+2}(\Omega^\bullet_{X/S}(D_v,D_h))$.
\end{lem}

\begin{proof}
We only deal with the case for $\Omega^\bullet_{X/S}(\log D)$.
The proof divides into three steps.

(Step 1: $i=1,j=0$) Recall the definition of $\phi^1:\Omega^1_{X'/Y}
(\log D'/E_a)[-1]\ra F_*\Omega^\bullet_{X/Y}(\log D/E_a)$ given in
Theorem \ref{2.2}. We choose an \'etale covering $\UU=(U_i)_{i\in I}$
of $X$, and a lifting $\wt{F}_i:\wt{U}_i\ra\wt{U}'_i$ of $F$ compatible
with $\wt{D}$ for each $i\in I$. On $U'_i$, we take
\[
f_i=\wt{F}^*_i/p:\Omega^1_{X'/Y}(\log D'/E_a)|_{U'_i}
\ra F_*\Omega^1_{X/Y}(\log D/E_a)|_{U'_i}.
\]
On $U'_{ij}=U'_i\cap U'_j$, we take
\[
h_{ij}=(\wt{F}^*_j-\wt{F}^*_i)/p:\Omega^1_{X'/Y}(\log D'/E_a)|_{U'_{ij}}
\ra F_*\OO_X|_{U'_{ij}}.
\]
We have the following relations:
\[
df_i=0,\,\,f_j-f_i=dh_{ij}\,\,\hbox{and}\,\,h_{ij}+h_{jk}=h_{ik}.
\]
The morphism $\phi^1$ is the composition of $u=(h_{ij},f_i):\Omega^1_{X'/Y}
(\log D'/E_a)[-1]\ra F_*\check{\CC}(\UU,\Omega^\bullet_{X/Y}(\log D/E_a))$
and the inverse of the quasi-isomorphism $F_*\Omega^\bullet_{X/Y}
(\log D/E_a)\ra F_*\check{\CC}(\UU,\Omega^\bullet_{X/Y}(\log D/E_a))$.
By adjunction of $(F^*_{Y/S},F_{Y/S*})$ and abuse of notation,
we have a morphism
\begin{equation*}
u:\Omega^1_{X_1/Y_1}(\log D_1/E_{a1})[-1]\ra F_{X/S*}\check{\CC}(\UU,
\Omega^\bullet_{X/Y}(\log D/E_a)).
\end{equation*}

Similarly, the liftings $\wt{U}_i\stackrel{\wt{F}_i}{\ra}\wt{U}'_i\ra
(\wt{U}_i)_1$ of $F_{X/S}$ provide a morphism
\[
u_1:\Omega^1_{X_1/S}(\log D_1)[-1]\ra F_{X/S*}\check{\CC}
(\UU,\Omega^\bullet_{X/S}(\log D)).
\]
Since $\check{\CC}(\UU,\Omega^\bullet_{X/S}(\log D))$ coincides with
$\check{\CC}(\UU,\Omega^\bullet_{X/S}(\log D)/K^2)$ in degree $\leq 1$,
we can consider $u_1$ with values in $F_{X/S*}\check{\CC}
(\UU,\Omega^\bullet_{X/S}(\log D)/K^2)$ to obtain a morphism
\[
u_1:\Omega^1_{X_1/S}(\log D_1)[-1]\ra F_{X/S*}\check{\CC}
(\UU,\Omega^\bullet_{X/S}(\log D)/K^2).
\]

Finally, we take $v=\wt{F}^*_{Y/S}/p:\Omega^1_{Y_1/S}(\log E_{a1})\ra
F_{Y/S*}\Omega^1_{Y/S}(\log E_a)$. By adjunction of $(F^*_{Y/S},F_{Y/S*})$
and abuse of notation, we have a homomorphism
\[
v:F^*_{Y/S}\Omega^1_{Y_1/S}(\log E_{a1})\ra\Omega^1_{Y/S}(\log E_a).
\]
Applying $f^*$ to the above homomorphism and using the
commutativity of (\ref{f5.27}), we have a homomorphism
\[
v:F^*_{X/S}f^*_1\Omega^1_{Y_1/S}(\log E_{a1})
\ra f^*\Omega^1_{Y/S}(\log E_a).
\]
By adjunction of $(F^*_{X/S},F_{X/S*})$ and the composition with
a natural morphism, we have a morphism
\begin{eqnarray*}
v:f^*_1\Omega^1_{Y_1/S}(\log E_{a1})[-1] &\ra& F_{X/S*}(f^*\Omega^1_{Y/S}
(\log E_a)\otimes\check{\CC}(\UU,\Omega^{\bullet-1}_{X/Y}(\log D/E_a)))
\nonumber \\
&=& F_{X/S*}\check{\CC}(\UU,f^*\Omega^1_{Y/S}(\log E_a)\otimes\Omega^
{\bullet-1}_{X/Y}(\log D/E_a)).
\end{eqnarray*}

We shall prove that $v$, $u_1$ and $u$ fit into the following
commutative diagram with exact rows:
\[\tiny
\xymatrix{
0 \ar[r] & f^*_1\Omega^1_{Y_1/S}(\log E_{a1})[-1] \ar[r] \ar[d]^v &
\Omega^1_{X_1/S}(\log D_1)[-1] \ar[r] \ar[d]^{u_1} &
\Omega^1_{X_1/Y_1}(\log D_1/E_{a1})[-1] \ar[r] \ar[d]^{u} & 0 \\
0 \ar[r] & F_{X/S*}\check{\CC}(\UU,f^*\Omega^1_{Y/S}(\log E_a)\otimes
\Omega^{\bullet-1}_{X/Y}(\log D/E_a)) \ar[r] &
F_{X/S*}\check{\CC}(\UU,\Omega^\bullet_{X/S}(\log D)/K^2) \ar[r] &
F_{X/S*}\check{\CC}(\UU,\Omega^\bullet_{X/Y}(\log D/E_a)) \ar[r] & 0 \\
0 \ar[r] & F_{X/S*}(f^*\Omega^1_{Y/S}(\log E_a)\otimes
\Omega^{\bullet-1}_{X/Y}(\log D/E_a)) \ar[r] \ar[u]_{q.i.} &
F_{X/S*}(\Omega^\bullet_{X/S}(\log D)/K^2) \ar[r] \ar[u]_{q.i.} &
F_{X/S*}\Omega^\bullet_{X/Y}(\log D/E_a) \ar[r] \ar[u]_{q.i.} & 0.
}
\]
Since $v=(0,f^Y_i)$, $u_1=(h_{ij}^X,f^X_i)$ and $u=(h_{ij},f_i)$, the
upper diagram is commutative. The lower one is a quasi-isomorphism of
short exact sequences of complexes. The morphism of distinguished
triangles defined by this diagram gives the commutativity of the
diagram in Lemma \ref{5.8} for the case $i=1,j=0$.

(Step 2: $j=0$) Recall that $\phi^i$ ($1\leq i<p$) is deduced from
$\phi^1$ by the composition:
{\scriptsize
\[
\Omega^i_{X'/Y}(\log D'/E_a)[-i]\stackrel{a}{\ra}\Omega^1_{X'/Y}(\log D'/E_a)
^{\otimes i}[-i]\stackrel{(\phi^1)^{\otimes i}}{\ra}(F_*\Omega^\bullet_{X/Y}
(\log D/E_a))^{\otimes i}\stackrel{\pi}{\ra}F_*\Omega^\bullet_{X/Y}
(\log D/E_a),
\]}
where $\pi$ is the product map and $a$ is the antisymmetrization map
\[
a(x_1\wedge\cdots\wedge x_i)=\frac{1}{i!}\sum_{\sigma\in S_i}\sgn(\sigma)
x_{\sigma(1)}\wedge\cdots\wedge x_{\sigma(i)}.
\]
Consider the following diagram:
\[
\xymatrix{
\Omega^i_{X_1/Y_1}(\log D_1/E_{a1})[-i] \ar[r] \ar[d]_a &
f^*_1\Omega^1_{Y_1/S}(\log E_{a1})\otimes\Omega^{i-1}_{X_1/Y_1}
(\log D_1/E_{a1})[-i+1] \ar[d]^{1\otimes a} \\
\Omega^1_{X_1/Y_1}(\log D_1/E_{a1})^{\otimes i}[-i] \ar[d]_{(\phi^1)^
{\otimes i}} \ar[r] &
f^*_1\Omega^1_{Y_1/S}(\log E_{a1})\otimes(\Omega^1_{X_1/Y_1}(\log D_1/E_{a1}))
^{\otimes i-1}[-i+1] \ar[d]^{1\otimes(\phi^1)^{\otimes i-1}} \\
(F_{X/S*}\Omega^\bullet_{X/Y}(\log D/E_a))^{\otimes i} \ar[d]_\pi
\ar[r] &
F_{X/S*}(f^*\Omega^1_{Y/S}(\log E_a)\otimes(\Omega^\bullet_{X/Y}(\log D/E_a))
^{\otimes i-1}) \ar[d]^\pi \\
F_{X/S*}\Omega^\bullet_{X/Y}(\log D/E_a) \ar[r] &
F_{X/S*}(f^*\Omega^1_{Y/S}(\log E_a)\otimes\Omega^\bullet_{X/Y}(\log D/E_a)),
}
\]
where the horizontal morphisms from top to bottom are $\partial$, $\sum
(1\otimes\cdots\otimes\partial\otimes\cdots\otimes 1)$, $\sum F_{X/S*}
(1\otimes\cdots\otimes\partial\otimes\cdots\otimes 1)$ and $F_{X/S*}\partial$ respectively.
Since the map $a$ is compatible with the Koszul filtration, we obtain the
commutativity of the upper diagram. The commutativity of the mediate one
follows from (Step 1), and the commutativity of the lower one follows from
Lemma \ref{5.7}.

(Step 3: general case) It follows from (Step 2) since $\Omega^\bullet_{Y/S}
(\log E_a)(\bH^!)$ is a graded differential module over $\Omega^\bullet_{Y/S}
(\log E_a)$.
\end{proof}

\begin{proof}[Proof of Proposition \ref{5.6}]
By applying $\oplus_k\bR^kf_{1*}$ to the diagrams in Lemma \ref{5.8},
we obtain the commutativity of the diagram (\ref{f5.39}).
By definition, the image of $v^j\otimes u^{i-j}$ is already contained
in $F_{Y/S*}(\Omega^j_{Y/S}(\log E_a)\otimes\Fil_{i-j}\bH^!)$.
Since (\ref{f5.39}) is commutative, we have
\[
\nabla(\im(v^j\otimes u^{i-j}))\subset\im(v^{j+1}\otimes u^{i-j-1})
\subset F_{Y/S*}(\Omega^{j+1}_{Y/S}(\log E_a)\otimes\Fil_{i-j-1}\bH^!),
\]
hence $\im(v^j\otimes u^{i-j})\subset G_iF_{Y/S*}\Omega^\bullet_{Y/S}
(\log E_a)(\bH^!)$ by the definition of $G_i$.
By construction, $\gr^i\Omega^\bullet_{Y_1/S}(\log E_{a1})(\bH^!_1)
\ra\gr^G_iF_{Y/S*}\Omega^\bullet_{Y/S}(\log E_a)(\bH^!)$ is a
quasi-isomorphism, which implies that (\ref{f5.41}) is a
quasi-isomorphism.
\end{proof}

Furthermore, we can eliminate the hypothesis (iii) in Assumption \ref{5.4}
to obtain the following main theorem in this paper for $!=$ null or $\dag$.

\begin{thm}\label{5.9}
Let $f:X\ra Y$ be an $E$-semistable $S$-morphism with an adapted divisor
$D$, and $\wt{f}:(\wt{X},\wt{D})\ra(\wt{Y},\wt{E}_a)$ a lifting of $f:(X,D)
\ra(Y,E_a)$ over $\wt{S}$. Assume that $f$ is proper and $\dim(X/Y)\leq p$.
Then for any $i<p$, we have a morphism in $D(Y_1)$:
\begin{equation}
\phi^i=\phi^i_{\wt{f}}:\gr^i\Omega^\bullet_{Y_1/S}(\log E_{a1})(\bH^!_1)
\ra G_iF_{Y/S*}\Omega^\bullet_{Y/S}(\log E_a)(\bH^!), \label{f5.44}
\end{equation}
such that the composition of (\ref{f5.44}) with the projection onto
$\gr^G_i$ is a quasi-isomorphism, which is the Cartier isomorphism
(\ref{f5.31}).

Furthermore, for any $i<p$, we have an isomorphism in $D(Y_1)$:
\begin{equation}
\phi=\sum_{j\leq i}\phi^j:\bigoplus_{j\leq i}\gr^j\Omega^\bullet_{Y_1/S}
(\log E_{a1})(\bH^!_1)\stackrel{\sim}{\ra}G_iF_{Y/S*}\Omega^\bullet_{Y/S}
(\log E_a)(\bH^!). \label{f5.45}
\end{equation}
For any $i>\dim(X/S)$, we have $\gr^i\Omega^\bullet_{Y/S}(\log E_a)
(\bH^!)=0$. Consequently, if $\dim(X/S)<p$, then (\ref{f5.45})
gives rise to a decomposition in $D(Y_1)$:
\begin{equation}
\phi:\bigoplus_i\gr^i\Omega^\bullet_{Y_1/S}(\log E_{a1})(\bH^!_1)
\stackrel{\sim}{\ra}F_{Y/S*}\Omega^\bullet_{Y/S}(\log E_a)(\bH^!).
\label{f5.46}
\end{equation}
\end{thm}

\begin{proof}
Since a lifting of the relative Frobenius morphism $F_{Y/S}:Y\ra Y_1$
always exists locally, Proposition \ref{5.6} is indeed a local version
of Theorem \ref{5.9}. The idea of the proof is to use Proposition \ref{5.6}
to obtain a coherent system of local splittings for $G_{p-1}F_{Y/S*}\Omega^\bullet
_{Y/S}(\log E_a)(\bH^!)$.

Take an \'etale covering $\UU=(U_i)_{i\in I}$ of $Y$.
By Proposition \ref{5.6}, on $U_i$ for any $i\in I$, there is a splitting $(v_i,u_i)$
of $G_{p-1}F_{Y/S*}\Omega^\bullet_{Y/S}(\log E_a)(\bH^!)|_{U_i}$, where $v_i,u_i$ are
defined as in (\ref{f5.36}) and (\ref{f5.37}) (see \cite[4.19]{il90} for the notion
``splitting''). By a similar argument to that of \cite[Proposition 4.3]{il90},
on $U_{ij}=U_i\cap U_j$ for any pair $(i,j)$, there is a homomorphism
$h_{ij}:\Omega^1_{Y_1/S}(\log E_{a1})|_{U_{ij1}}\ra F_{Y/S*}\OO_Y|_{U_{ij}}$,
such that the following conditions hold:
\begin{eqnarray*}
v^1_j-v^1_i &=& dh_{ij}\quad\hbox{on}\,\,U_{ij}, \\
h_{ij}+h_{jk} &=& h_{ik}\quad\hbox{on}\,\,U_{ijk}, \\
u^n_j-u^n_i &=& (u_ih_{ij})\circ d\quad\hbox{on}\,\,U_{ij}\,\,\hbox{for}\,\,n<p,
\end{eqnarray*}
where in the third equality, $d:\gr^n\bH^!_1\ra \oplus_{0<m\leq n}\gr^{n-m}\bH^!_1\otimes
\Gamma^m\Omega^1_{Y_1/S}(\log E_{a1})$ is the differential map of the complex $NC(\gr^\bullet\bH^!_1)$
defined as in \cite[(4.1.7)]{il90}, $u_ih_{ij}$ is given by $(u_ih_{ij})(x\otimes a)=u_i(x)h_{ij}(a)$,
and the map $h_{ij}:\Gamma^m\Omega^1_{Y_1/S}(\log E_{a1})\ra F_{Y/S*}\OO_Y$ on $U_{ij}$ is defined by
the polynomial map $x^{[m]}\mapsto h_{ij}(x)^m/m!$ for $x\in\Omega^1_{Y_1/S}(\log E_{a1})$ (note that
$\Gamma\Omega^1_{Y_1/S}(\log E_{a1})$ is the divided power algebra of $\Omega^1_{Y_1/S}(\log E_{a1})$).

Thus $G_{p-1}F_{Y/S*}\Omega^\bullet_{Y/S}(\log E_a)(\bH^!)$ has a coherent system of local splittings
$(\UU=(U_i),(v_i),(u_i),(h_{ij}))$. It follows from \cite[Theorem 4.20]{il90} that there exist morphisms
$\phi^i$ (\ref{f5.44}) satisfying all of the required properties.
\end{proof}

\begin{cor}\label{5.10}
Let $f:X\ra Y$ be an $E$-semistable $S$-morphism with an adapted
divisor $D$. Assume that $f$ is proper and $g:Y\ra S$ is proper.
Assume that $f:(X,D)\ra(Y,E_a)$ has a lifting $\wt{f}:(\wt{X},\wt{D})
\ra(\wt{Y},\wt{E}_a)$ over $\wt{S}$ and $\dim(X/S)<p$.
Then the Hodge spectral sequence for $\Omega^\bullet_{Y/S}
(\log E_a)(\bH^!)$ and $\bR g_*$ degenerates in $E_1$:
\[
E^{ij}_1=\bR^{i+j}g_*\gr^i\Omega^\bullet_{Y/S}(\log E_a)(\bH^!)
\Rightarrow \bR^{i+j}g_*\Omega^\bullet_{Y/S}(\log E_a)(\bH^!),
\]
and each $E^{ij}_1$ is locally free of finite type, and of formation
compatible with any base change.
\end{cor}

\begin{proof}
We can use the decomposition (\ref{f5.46}) and an analogous argument
to \cite[4.1.2]{di} to complete the proof.
\end{proof}

\begin{cor}\label{5.11}
Let $K$ be a field of characteristic zero, $S=\Spec K$, $X,Y$ proper and
smooth $S$-schemes, and $f:X\ra Y$ an $E$-semistable $S$-morphism with
an adapted divisor $D$. Then the Hodge spectral sequence for $\Omega
^\bullet_{Y/S}(\log E_a)(\bH^!)$ degenerates in $E_1$:
\[
E^{ij}_1=\cH^{i+j}(Y,\gr^i\Omega^\bullet_{Y/S}(\log E_a)(\bH^!))
\Rightarrow \cH^{i+j}(Y,\Omega^\bullet_{Y/S}(\log E_a)(\bH^!)).
\]
\end{cor}

\begin{proof}
It follows from Corollary \ref{5.10} and the standard argument using
the reduction modulo $p$ technique (see \cite[2.7]{di} and \cite{il96}).
\end{proof}

\section{Applications to vanishing theorems}\label{S6}

In this section, let $k$ be a perfect field of characteristic $p>0$,
and $W_2(k)$ the ring of Witt vectors of length two of $k$. There
are some applications of the main theorem to vanishing theorems.

\begin{thm}\label{6.1}
Let $S=\Spec k$, $\wt{S}=\Spec W_2(k)$, and $X,Y$ proper and smooth
$S$-schemes. Let $f:X\ra Y$ be an $E$-semistable $S$-morphism with
an adapted divisor $D$, and $\LL$ an ample invertible sheaf on $Y$.
Assume that $f:(X,D)\ra(Y,E_a)$ has a lifting $\wt{f}:(\wt{X},\wt{D})
\ra(\wt{Y},\wt{E}_a)$ over $\wt{S}$ and $\dim(X/S)<p$. Then we have
\begin{eqnarray}
\cH^{i+j}(Y,\LL\otimes\gr^i\Omega^\bullet_{Y/S}(\log E_a)(\bH^!))=0
\,\,\hbox{for any}\,\,i+j>\dim(Y/S), \label{f6.1} \\
\cH^{i+j}(Y,\LL^{-1}\otimes\gr^i\Omega^\bullet_{Y/S}(\log E_a)(\bH^!))=0
\,\,\hbox{for any}\,\,i+j<\dim(Y/S). \label{f6.2}
\end{eqnarray}
\end{thm}

\begin{proof}
We use an analogous argument to those of \cite[2.8]{di} and
\cite[Corollary 4.16]{il90}. Let $\MM$ be an invertible sheaf on $Y$.
Define
\[
h^{ij}(\MM)=\dim\cH^{i+j}(Y,\MM\otimes\gr^i\Omega^\bullet_{Y/S}
(\log E_a)(\bH^!)).
\]
Then for all $n$, we have
\begin{equation}
\sum_{i+j=n}h^{ij}(\MM)\leq\sum_{i+j=n}h^{ij}(\MM^p). \label{f6.3}
\end{equation}
Indeed, denote by $\MM_1$ the inverse image of $\MM$ on $Y_1$, then we have
$\MM^p=F_{Y/S}^*\MM_1$. The Hodge spectral sequence
\begin{eqnarray*}
E^{ij}_1 &=& \cH^{i+j}(Y_1,\MM_1\otimes F_{Y/S*}\gr^i\Omega^\bullet_{Y/S}
(\log E_a)(\bH^!)) \\
&\Rightarrow& \cH^{i+j}(Y_1,\MM_1\otimes F_{Y/S*}\Omega
^\bullet_{Y/S}(\log E_a)(\bH^!))
\end{eqnarray*}
gives rise to the inequality
\[
\dim \cH^n(Y_1,\MM_1\otimes F_{Y/S*}\Omega^\bullet_{Y/S}(\log E_a)(\bH^!))
\leq\sum_{i+j=n}h^{ij}(\MM^p).
\]
On the other hand, by the decomposition (\ref{f5.46}), we have
\begin{eqnarray*}
&& \dim \cH^n(Y_1,\MM_1\otimes F_{Y/S*}\Omega^\bullet_{Y/S}(\log E_a)
(\bH^!)) \\
&=& \sum_{i+j=n} \dim \cH^{i+j}(Y_1,\MM_1\otimes\gr^i\Omega^\bullet
_{Y_1/S}(\log E_{a1})(\bH^!_1))=\sum_{i+j=n}h^{ij}(\MM),
\end{eqnarray*}
which proves (\ref{f6.3}).

Next, we shall prove $h^{ij}(\LL^{p^N})=0$ for $N$ sufficiently large and
for all $i+j>\dim(Y/S)$. The stupid filtration of the Kodaira-Spencer complex
\[
\gr^i\Omega^\bullet_{Y/S}(\log E_a)(\bH^!)=(\gr^i\bH^!\stackrel{\nabla}{\ra}
\Omega^1_{Y/S}(\log E_a)\otimes\gr^{i-1}\bH^!\stackrel{\nabla}{\ra}\cdots)
\]
gives rise to the following spectral sequence
\begin{eqnarray*}
E^{rs}_1=H^s(Y,\LL^{p^N}\otimes\Omega^r_{Y/S}(\log E_a)\otimes\gr^{i-r}\bH^!)
\Rightarrow\cH^{r+s}(Y,\LL^{p^N}\otimes\gr^i\Omega^\bullet_{Y/S}(\log E_a)(\bH^!)).
\end{eqnarray*}
We focus on those terms with $r+s=i+j>\dim(Y/S)$.
If $s=0$ then $\Omega^r_{Y/S}(\log E_a)=0$.
If $s>0$ then the choices of $r$ and $s$ are finite. By the Serre vanishing theorem,
we can choose $N$ sufficiently large such that $E^{rs}_1=0$ for all $r$ and $s$,
hence $h^{ij}(\LL^{p^N})=0$ holds for all $i+j>\dim(Y/S)$. Thanks to
(\ref{f6.3}), we obtain the vanishing (\ref{f6.1}).

By a similar argument, we can prove $h^{ij}(\LL^{-p^N})=0$ for $N$
sufficiently large and for all $i+j<\dim(Y/S)$. Thanks to (\ref{f6.3}),
we obtain the vanishing (\ref{f6.2}).
\end{proof}

\begin{cor}\label{6.2}
Let $K$ be a field of characteristic zero, $S=\Spec K$, and $X,Y$ proper
and smooth $S$-schemes. Let $f:X\ra Y$ be an $E$-semistable $S$-morphism
with an adapted divisor $D$, and $\LL$ an ample invertible sheaf on $Y$.
Then we have
\begin{eqnarray*}
\cH^{i+j}(Y,\LL\otimes\gr^i\Omega^\bullet_{Y/S}(\log E_a)(\bH^!))=0
\,\,\hbox{for any}\,\,i+j>\dim(Y/S), \\
\cH^{i+j}(Y,\LL^{-1}\otimes\gr^i\Omega^\bullet_{Y/S}(\log E_a)(\bH^!))=0
\,\,\hbox{for any}\,\,i+j<\dim(Y/S).
\end{eqnarray*}
\end{cor}

\begin{proof}
It follows from Theorem \ref{6.1} and the reduction modulo $p$
technique.
\end{proof}

\begin{thm}\label{6.3}
Let $S=\Spec k$, $\wt{S}=\Spec W_2(k)$ and $X$, $Y$ proper and smooth
$S$-schemes. Let $f:X\ra Y$ be an $E$-semistable $S$-morphism with an
adapted divisor $D$, and $\LL$ an ample invertible sheaf on $Y$.
Assume that $f:(X,D)\ra(Y,E_a)$ has a lifting $\wt{f}:(\wt{X},\wt{D})
\ra(\wt{Y},\wt{E}_a)$ over $\wt{S}$ and $\dim(X/S)<p$. Then we have
\[
H^i(Y,\LL\otimes R^jf_*\omega_{X/S}(D))=0\,\,\hbox{and}\,\,
H^i(Y,\LL\otimes R^jf_*\omega_{X/S}(D_h))=0
\]
for any $i>0$ and $j\geq 0$.
\end{thm}

\begin{proof}
Suppose $\dim(X/S)=n$, $\dim(X/Y)=d$ and $\dim(Y/S)=e$.

(1) Consider $\gr^n\Omega^\bullet_{Y/S}(\log E_a)(\bH)$, whose $k$-th
component is
\[
\Omega^k_{Y/S}(\log E_a)\otimes\gr^{n-k}\bH=\Omega^k_{Y/S}(\log E_a)\otimes
(\oplus_l R^{l-n+k}f_*\Omega^{n-k}_{X/Y}(\log D/E_a)).
\]
Since $\Omega^k_{Y/S}(\log E_a)=0$ for any $k>e$ and $\Omega^{n-k}_{X/Y}
(\log D/E_a)=0$ for any $k<e$, we have
\begin{eqnarray*}
\gr^n\Omega^\bullet_{Y/S}(\log E_a)(\bH) &=& \omega_{Y/S}(E_a)\otimes
(\oplus_l R^{l-d}f_*\omega_{X/Y}(D_h))[-e] \\
&=& \oplus_{k\geq 0} R^kf_*\omega_{X/S}(D)[-e].
\end{eqnarray*}
In Theorem \ref{6.1}, taking $i=n$ and $r=i+j-e>0$, we have
\begin{eqnarray*}
\bigoplus_{k\geq 0}H^r(Y,\LL\otimes R^kf_*\omega_{X/S}(D))=0,\,\,
\hbox{i.e.} \\
H^r(Y,\LL\otimes R^kf_*\omega_{X/S}(D))=0\,\,\hbox{for any}\,\,r>0\,\,
\hbox{and}\,\,k\geq 0.
\end{eqnarray*}

(2)  Consider $\gr^n\Omega^\bullet_{Y/S}(\log E_a)(\bH^\dag)$, whose $k$-th
component is
\[
\Omega^k_{Y/S}(\log E_a)\otimes\gr^{n-k}\bH^\dag=\Omega^k_{Y/S}(\log E_a)
\otimes(\oplus_l R^{l-n+k}f_*\Omega^{n-k}_{X/Y}(D_v,D_h)).
\]
Since $\Omega^k_{Y/S}(\log E_a)=0$ for any $k>e$ and $\Omega^{n-k}_{X/Y}
(D_v,D_h)=0$ for any $k<e$, we have
\begin{eqnarray*}
\gr^n\Omega^\bullet_{Y/S}(\log E_a)(\bH^\dag) &=& \omega_{Y/S}(E_a)
\otimes(\oplus_l R^{l-d}f_*\omega_{X/Y}(D_h-D_v))[-e] \\
&=& \oplus_{k\geq 0} R^kf_*\omega_{X/S}(D_h)[-e].
\end{eqnarray*}
In Theorem \ref{6.1}, taking $i=n$ and $r=i+j-e>0$, we have
\begin{eqnarray}
\bigoplus_{k\geq 0}H^r(Y,\LL\otimes R^kf_*\omega_{X/S}(D_h))=0,\,\,\hbox{i.e.}
\nonumber \\
H^r(Y,\LL\otimes R^kf_*\omega_{X/S}(D_h))=0\,\,\hbox{for any}\,\,r>0\,\,
\hbox{and}\,\,k\geq 0. \label{f6.4}
\end{eqnarray}
\end{proof}

In order to give further applications, we need the following definition
\cite[Definition 2.3]{xie10}.

\begin{defn}\label{6.4}
Let $X$ be a smooth scheme over $k$. $X$ is said to be strongly
liftable over $W_2(k)$, if the following two conditions hold:

(i) $X$ is liftable over $W_2(k)$; and

(ii) there is a lifting $\wt{X}$ of $X$, such that for any prime
divisor $D$ on $X$, $(X,D)$ has a lifting $(\wt{X},\wt{D})$ over
$W_2(k)$, where $\wt{X}$ is fixed for all $D$.
\end{defn}

It was proved in \cite{xie10} and \cite{xie} that $\A^n_k$, $\PP^n_k$,
smooth projective curves, smooth projective rational surfaces, certain
smooth complete intersections in $\PP^n_k$, and smooth toric varieties
are strongly liftable over $W_2(k)$. As a consequence of Theorem
\ref{6.3}, we can obtain some vanishing results for certain strongly
liftable varieties.

\begin{cor}\label{6.5}
Let $X=X(\Delta,k)$ be a smooth projective toric variety associated to a fan
$\Delta$ with $\ch(k)=p>\dim X$, $Y$ a smooth projective variety over $k$,
$f:X\ra Y$ an $E$-semistable morphism with an adapted divisor $D$, and $\LL$
an ample invertible sheaf on $Y$. Then $f:(X,D)\ra (Y,E_a)$ has a lifting
$\wt{f}:(\wt{X},\wt{D})\ra (\wt{Y},\wt{E}_a)$ over $W_2(k)$. Consequently,
we have
\[
H^i(Y,\LL\otimes R^jf_*\omega_X(D))=0 \,\,\hbox{and}\,\,
H^i(Y,\LL\otimes R^jf_*\omega_X(D_h))=0
\]
for any $i>0$ and $j\geq 0$.
\end{cor}

\begin{proof}
Let $H$ be a general very ample effective divisor on $Y$ and $F=f^{-1}(H)$
the divisor on $X$. By \cite[Theorem 3.1]{xie}, $X$ is strongly liftable,
hence there are a lifting $\wt{X}=X(\Delta,W_2(k))$ of $X=X(\Delta,k)$
and a lifting $\wt{D}+\wt{F}\subset\wt{X}$ of $D+F\subset X$ over $W_2(k)$.
More precisely, let $G$ be a torus invariant divisor on $X$ determined by
the data $\{u(\sigma)\}\in\varprojlim M/M(\sigma)$ such that $G$ is linearly
equivalent to $F$. Then we can construct a torus invariant divisor $\wt{G}$
on $\wt{X}$ determined by the same data $\{u(\sigma)\}\in\varprojlim
M/M(\sigma)$ and prove that the natural map $H^0(\wt{X},\wt{G})\ra H^0(X,G)$
is surjective. Thus we can take a lifting $\wt{F}$ of $F$ such that $\wt{F}$
is linearly equivalent to $\wt{G}$.

By definition, the linear system $|F|$ is base point free, hence so is $|G|$.
By \cite[Page 68, Proposition]{fu}, the continuous piecewise linear function
$\psi_G$ on $|\Delta|$ defined in \cite[Page 66]{fu} is upper convex. Since
the functions $\psi_{\wt{G}}$ and $\psi_G$ are the same, $\psi_{\wt{G}}$ is
also upper convex. Thus the linear system $|\wt{G}|$ is base point free,
hence so is $|\wt{F}|$. Thus the linear system $|\wt{F}|$ defines
a $W_2(k)$-morphism $\wt{f}:\wt{X}\ra \wt{Y}$.

It is easy to verify that $\wt{Y}$ is a lifting of $Y$ and $\wt{f}$ is a
lifting of $f$ over $W_2(k)$. By \cite[Lemmas 8.13, 8.14]{ev} or
\cite[Lemma 2.2]{xie}, $\wt{D}$ is relatively simple normal crossing over
$W_2(k)$. Hence we can verify that $\wt{f}:\wt{X}\ra \wt{Y}$ is an
$\wt{E}$-semistable morphism and $\wt{D}$ is adapted to $\wt{f}$,
which imply that $\wt{f}:(\wt{X},\wt{D})\ra (\wt{Y},\wt{E}_a)$
is a lifting of $f:(X,D)\ra (Y,E_a)$ over $W_2(k)$.
By Theorem \ref{6.3}, we obtain the required vanishings.
\end{proof}

\begin{cor}\label{6.6}
Let $X$ be a smooth projective rational surface over $k$ with $\ch(k)=p>3$,
$f:X\ra\PP^1_k$ an $E$-semistable morphism with an adapted divisor $D$,
and $\LL$ an ample invertible sheaf on $\PP^1_k$. Then $f:(X,D)\ra (Y,E_a)$
has a lifting $\wt{f}:(\wt{X},\wt{D})\ra (\wt{Y},\wt{E}_a)$ over $W_2(k)$.
Consequently, we have
\[
H^i(\PP^1_k,\LL\otimes R^jf_*\omega_X(D))=0 \,\,\hbox{and}\,\,
H^i(\PP^1_k,\LL\otimes R^jf_*\omega_X(D_h))=0
\]
for any $i>0$ and $j\geq 0$.
\end{cor}

\begin{proof}
Let $P\in\PP^1_k$ be a general point and $F=f^{-1}(P)$ the fiber of $f$.
By \cite[Theorem 1.3]{xie10}, $X$ is strongly liftable, hence there are
a lifting $\wt{X}$ of $X$ and a lifting $\wt{D}+\wt{F}\subset\wt{X}$
of $D+F\subset X$ over $W_2(k)$. Since both $X$ and $\wt{X}$ are birational
to certain smooth projective toric surfaces through a sequence of blow-ups
along some closed points, by a similar argument to the proof of Corollary
\ref{6.5}, we can show that the linear system $|\wt{F}|$ is base point free,
which gives rise to a $W_2(k)$-morphism $\wt{f}:\wt{X}\ra \PP^1_{W_2(k)}$.
By \cite[Lemmas 8.13, 8.14]{ev} or \cite[Lemma 2.2]{xie}, it is easy to
verify that $\wt{f}:(\wt{X},\wt{D})\ra (\PP^1_{W_2(k)},\wt{E}_a)$
is a lifting of $f:(X,D)\ra (\PP^1_k,E_a)$ over $W_2(k)$.
By Theorem \ref{6.3}, we obtain the required vanishings.
\end{proof}

\textsc{School of Mathematical Sciences, Fudan University,
Shanghai 200433, China}

\textit{E-mail address}: \texttt{qhxie@fudan.edu.cn}

\end{document}